\documentclass[12pt]{amsart}
\usepackage{amssymb,amsmath,amsfonts,epsfig,latexsym,texdraw}
\usepackage[all]{xy}

 \newtheorem{theorem}{Theorem}[section]
\newtheorem{definition}[theorem]{Definition}

\newtheorem{proposition}[theorem]{Proposition}
\newtheorem{lemma}[theorem]{Lemma}
\newtheorem{corollary}[theorem]{Corollary} 
\newtheorem{question}[theorem]{Question}

\theoremstyle{definition}

\newtheorem{example}[theorem]{Example}

\def\N{\ensuremath{\mathbb{N}}}
\def\Z{\ensuremath{\mathbb{Z}}}
\def\Q{\ensuremath{\mathbb{Q}}}

\def\cp{\ensuremath{\mathcal{P}}}

\def\C{\ensuremath{\mathbb{C}}}
\def\G{\ensuremath{\mathbb{G}}}
\def\A{\ensuremath{\mathbb{A}}}
\def\R{\ensuremath{\mathbb{R}}}
\def\K{\ensuremath{\mathbb{K}}}
\def\k{\ensuremath{\mathbf{k}}}

\def\O{\ensuremath{\mathcal{O}}}

\def\cul{\ensuremath{\mathcal{L}}}
\def\cu{\ensuremath{\mathcal{U}}}

\def\cc{\ensuremath{\mathcal{C}}}
\def\cd{\ensuremath{\mathcal{D}}}
\def\ck{\ensuremath{\mathcal{K}}}
\def\cm{\ensuremath{\mathcal{M}}}
\def\cn{\ensuremath{\mathcal{N}}}

\def\tXi{\ensuremath{\widetilde{\Xi}}}
\def\gp{\ensuremath{\operatorname{gp}}}

\def\<{\ensuremath{\langle}}
\def\>{\ensuremath{\rangle}}

\DeclareMathOperator{\Gal}{Gal}
\DeclareMathOperator{\Aut}{Aut}
\DeclareMathOperator{\Div}{Div}

 \DeclareMathOperator{\Hom}{Hom}

\DeclareMathOperator{\Spec}{Spec}

\DeclareMathOperator{\Trop}{Trop} \DeclareMathOperator{\val}{val}

\DeclareMathOperator{\ord}{ord}
\DeclareMathOperator{\res}{res}

\DeclareMathOperator{\dist}{dist}
\DeclareMathOperator{\sm}{sm}
\DeclareMathOperator{\an}{an}
\def\DLog{\operatorname{dlog}}

\begin{document}

\title[Lifting Tropical Curves in Space]{Lifting Tropical Curves in Space and Linear Systems on Graphs}

\author{Eric Katz}
\address{Department of Mathematics, University of Texas at Austin, Austin, TX 78712}
\email{eekatz@math.utexas.edu}

\begin{abstract}
Tropicalization is a procedure for associating a polyhedral complex in Euclidean space to a subvariety of an algebraic torus.  We study the question of which graphs arise from tropicalizing algebraic curves.  By using Baker's specialization of linear systems from curves to graphs, we are able to give a necessary condition for a balanced weighted graph to be the tropicalization of a curve.  Our condition reproduces a generalization of Speyer's well-spacedness condition and also gives new conditions.    In addition, it suggests a new combinatorial structure on tropicalizations of algebraic curves.
\end{abstract}

\maketitle

\section{Introduction}

Tropical geometry transforms questions in algebraic geometry to ones in polyhedral geometry.  Let $\K=\C((t))$ be the field of formal Laurent series, and let $\O=\C[[t]]$ be its valuation ring, the ring of formal power series.  To a subvariety in an algebraic torus $V\subset(\K^*)^n$, the method of tropicalization \cite{EKL,Spe05} associates a polyhedral complex $\Trop(V)\subset\R^n$.  It is a natural question to ask which polyhedral complexes arise in this fashion.  In this case, we say that the polyhedral complex {\em lifts}.  The case where $V$ is a curve has been studied in papers of Speyer \cite{Spe07}, Nishinou \cite{N}, and Tyomkim \cite{Tyomkin} and in a forthcoming paper of Brugall\'{e} and Mikhalkin.  These papers give necessary and sufficient conditions for a tropical curve to lift.

It is natural to enlarge the class of objects studied from curves in $(\K^*)^n$ to maps of curves by using the approach of Nishinou-Siebert \cite{NS}.  Given a map of a smooth curve $f:C^*\rightarrow(\K^*)^n$, there is a {\em parameterized tropicalization} $\Trop(f):\Sigma\rightarrow\R^n$, a map of a graph that is rational affine linear on edges and maps vertices to rational points.  In this case, $\Trop(f)(\Sigma)=\Trop(f(C^*))$.  Here, the graph $\Sigma$ is a certain kind of dual graph of a regular semistable model completing $C^*$ over $\O$.  To each edge $e$ of $\Sigma$ is associated a weight $\mu(e)\in\N$ that satisfies the balancing condition: for each vertex $v\in V(\Sigma)$ with adjacent edges $e_1,\dots,e_k$ in primitive integer vector directions $w_1,\dots,w_k$,
\[\sum \mu(e_i)w_i=0.\]
Therefore, one should direct one's attention to balanced weighted parameterized graphs $\Phi:\Sigma\rightarrow\R^n$.  There is an important dichotomy introduced by Mikhalkin \cite{M05} between regular and super-abundant parameterized graphs depending on whether they move in a family of the expected dimension.  In fact, one may do a dimension count of graphs that have the same combinatorial type and edge directions by varying the length of bounded edges subject to the constraint that loops must close up.  More precisely, if $E(\Sigma)^\bullet$ is the set of bounded edges of $\Sigma$ then there is a natural map
\[\R^{E(\Sigma)^\bullet}\rightarrow\Hom(H_1(\Sigma),\R^n)\]
taking a choice of edge-lengths to a function associating a cycle $\gamma\in H_1(\Sigma)$ to the total displacement when traveling around the edges of $\gamma$.  The kernel of this map intersected with the positive orthant, $\R_+^{E(\Sigma)^\bullet}$ is the
space of graphs with the same combinatorial type and edge-directions.  If this map is surjective then $(\Sigma,\Phi)$ is
said to be {\em regular}.  Otherwise, it is said to be superabundant.  If a curve is regular and the residue field $\k$ has characteristic $0$, then it lifts by a theorem proved by Speyer \cite{Spe05} (see also the discussion in \cite{Mi06} and proofs by Nishinou \cite{N} and Tyomkin \cite{Tyomkin} using deformation theory).
A particular case where a curve is superabundant is when a cycle $\Gamma$ is mapped to a proper affine subspace $H$ of $\R^n$.   In this case, the map
\[\R^{E(\Sigma)^\bullet}\rightarrow\Hom(H_1(\Sigma),\R^n)\rightarrow \Hom(H_1(\Gamma),\R^n)\]
is not surjective.  We say the curve is {\em planar-superabundant}.   In the case of genus $1$ curves, every superabundant curve is planar-superabundant.
 
For genus $1$ balanced weighted integral graphs all of whose vertices have degree $2$ or $3$, there is a necessary and sufficient condition due to Speyer \cite{Spe07} for a superabundant graph to lift.  Let $\Gamma$ be the cycle in $\Gamma'=\Trop(f)^{-1}(H)$.  Let $x_1,\dots,x_s$ be the boundary points of the connected component of $\Gamma'$ containing $\Gamma$.  Then the minimum lattice distance $\dist(x_i,\Gamma)$ must be achieved at least twice.  Speyer proves this using uniformization of curves.  Nishinou \cite{N} later gave a proof of the higher genus analog of this condition using log deformation theory while Baker-Payne-Rabinoff \cite{BPR} gave a proof of Speyer's condition using the group law on the Berkovich elliptic curve.  A paper giving an analytic proof of the higher genus analog is in progress by Brugall\'{e}-Mikhalkin.

In this paper, we give a necessary condition for planar-superabundant graphs to lift which is equivalent to Speyer's and Nishinou's conditions when they apply and which gives a new condition in additional cases.   Our technique is to use the specialization of linear systems from curves to graphs developed by Baker \cite{B}.  One considers a particular  log $1$-form on $C$,  $\omega_m=f^*\frac{\text{d}z^m}{z^m}$  where $z^m$ is a character of $(\K^*)^n$.  By standard results from log geometry, this $1$-form extends to a regular log $1$-form on a regular semistable model $\cc$ over $\O$ completing $C^*$.  $\Sigma$ is the dual graph of $\cc$.  The degree of vanishing of $\omega_m$ on components $C_v$ of the central fiber $\cc_0$ defines a piecewise-linear function $\varphi_m:\Sigma\rightarrow \R\cup\{\infty\}$ which is an element of the linear system associated to the canonical bundle, $L(K_\Sigma)$.  The values of $\varphi_m$ constrains the poles and zeros of the restriction of $(\omega_m)_v=\frac{\omega_m}{t^h}|_{C_v}$ for $v\in\Gamma$ and a suitable value of $h$.  For certain $v$ depending on $m$, $\frac{\omega_m}{t^h}|_{C_v}$ is an exact $1$-form.  This gives very strong conditions on $\varphi_m$. For example, we make use of the facts that an exact rational $1$-form does not have simple poles and that a non-zero exact $1$-form on a cycle of rational curves (which is a degenerate elliptic curve) must have two poles counted with multiplicity. 
In the case where $\cc_0$ has all components rational and all vertices of $\Sigma$ have degree $2$ or $3$, the constraints are purely combinatorial.  Let $\Sigma^\bullet$ be the subgraph of $\Sigma$ consisting of all vertices and bounded edges.
The main result of this paper is the following:

\begin{theorem} \label{maintheorem} Let $f:C^*\rightarrow(\K^*)^n$ be a map of a smooth curve with parameterized tropicalization $\Trop(f):\Sigma\rightarrow \R^n$.  After possibly subdividing each edge of $\Sigma$ into $l$ congruent segments, 
for any $m\in M=\Hom((\K^*)^n,\K^*)$, there exists a non-negative piecewise-linear function $\varphi_m$ on $\Sigma$ with integer slopes.   $\varphi_m$ satisfies the following properties:
\begin{enumerate}
\item $\varphi_m$ is in the linear system $L(K_{\Sigma})$,  that is, $\Delta(\varphi_m)+K_{\Sigma}\geq 0$,
\item for $e\in E(\Sigma)$ with $m\cdot e\neq 0$, $\varphi_m=0$ on $e$,
\item for $e\in E(\Sigma)$ with $m\cdot e=0$, $\varphi_m$ never has slope 0 on $e$, and
\item for any $c\in\R$, set $H=\{x|\<m,x\>=c\}$; let $\Gamma'=\Trop(f)^{-1}(H)$, considered as a subspace of $\Sigma$, and  let $\Gamma$ be a bounded connected subgraph contained in the interior of $\Gamma'$; then $\varphi_m$ is $\cc_0$-ample on $\Gamma$ in $\Gamma'$.
\end{enumerate}
Moreover, the map $m\mapsto \varphi_m|_{\Sigma^\bullet}$ gives a tropical homomorphism of $M$ to $L(K_{\Sigma})|_{\Sigma^\bullet}$. 
\end{theorem}
The definitions of the terms in the statement are given in section 3.  This theorem is stated in maximum generality.  In many examples, the existence of $\varphi_m$ will only be obstructed for values of $m$ normal to an affine subspace $H$ containing a cycle of $\Sigma$.
  
We may rephrase this theorem as an obstruction to a balanced weighted integral graph $\Sigma'$ in $\R^n$ to be $\Trop(C)$ for a curve $C\subset (\K^*)^n$ in terms of the {\em tropical parameterizations} of Definition \ref{d:tparameterization}.

\begin{corollary} \label{c:istrop} If $\Sigma'=\Trop(C^*)$ then there is a tropical parameterization $p:\Sigma\rightarrow\Sigma'$ such that after a possible $l$-fold subdivision of $\Sigma$, there is a map $M\rightarrow L(K_\Sigma)$ given $m\mapsto\varphi_m$ satisfying the properties above.    \end{corollary}
 
We are able to derive the necessity of Speyer's well-spacedness condition from the above theorem.  
We also prove a weak version of well-spacedness suggested to us by Sam Payne that puts fewer constraints on the combinatorial type of $\Sigma$:

\begin{theorem} \label{weakwellspaced} 
Let $\Trop(f):\Sigma\rightarrow\R^n$ be the parameterized tropicalization of a map of a smooth curve with maximally degenerate reduction.  Let $m\in M$, $c\in\R$ and $H=\{x|\<m,x\>=c\}$, and let $\Gamma'$ be a component of $\Trop(f)^{-1}(H)$.  If $h^1(\Gamma')>0$ then, $\partial\Gamma'$ does not consist of a single trivalent vertex of $\Sigma$.
\end{theorem}

For many balanced weighted integral graph $\Sigma'$ in $\R^n$, there are only finitely many  tropical parameterizations $p:\Sigma\rightarrow\Sigma'$ of a given genus $g$, so only finitely many cases need to be checked to see if our conditions prevent a graph from being the tropicalization of a curve of genus $g$. 
  
 We ask whether following partial converse is true:
 \begin{question} Let $\Phi:\Sigma\rightarrow\R^n$ be a balanced weighted parameterized graph all of whose vertices have degree $2$ or $3$.  Suppose there is an association $m\mapsto\varphi_m$ satisfying the above properties.  Does there exists a map of a smooth curve $f:C^*\rightarrow (\K^*)^n$ with parameterized tropicalization $\Trop(f)?$
 \end{question}
 As our condition only appears to be sensitive to the planar-superabundant case, the general situation may be richer than the above question implies.
 
A goal of this paper is to study the combinatorial content of deformation theory.  In a certain sense, our approach is very similar to that of Nishinou and Tyomkin.  We were curious about how the non-vanishing of an obstruction at a certain order in deformation theory would manifest itself combinatorially.  In fact, we unwound the definition of the obstruction in the log obstruction group $H^1(\cc_0,\mathcal{H}\text{om}(f^*\Omega^1_{X^\dagger/\O^\dagger},\Omega^1_{{\cc^\dagger_0/\O^\dagger}}))$.
The obstructions that we were seeing looked very much like the conditions on the rank of a linear system on a graph as developed by Baker-Norine \cite{BN}.  Ultimately, however, they did not exactly fit into that framework and instead had to do with whether a particular combinatorially defined divisor $D_\varphi$ on a nodal curve had a non-trivial linear system.    While the statement of our condition is unwieldy, we believe it to be natural.  It is derived from obstructions to a piecewise linear function on $\Sigma$ to be the orders of vanishing of a rational function $s$ on $\cc$ at components of the central fiber.  Specifically, there must be an additional piecewise linear function that encodes the orders of vanishing of the log differential of $s$.
This work can also be seen in light of the recent work of Baker, Payne, and Rabinoff on Berkovich curves \cite{BPR}: our conditions constrains piecewise linear functions on the analytification $(C^*)^{\an}$ that can appear as $\log|s|$ for rational functions $s$ on $C^*$.
We hope that our techniques can be used in other problems.  

The functions $\varphi_m$ give an additional structure on tropical curves that arise as tropicalizations.  In future work, we hope to explore the analogous structure on higher dimensional tropicalizations.  We expect this work to fit into the log geometry framework of Gross-Siebert \cite{GS2}.

A slightly weakened version of our main theorem holds for valuation fields $\K$ of equicharacteristic $0$.  For a given $m\in M$, one may produce $\varphi_m$ satisfying the conditions of our theorem on some subdivision of $\Sigma$.  This subdivision however may depend on $m$.

Our method of understanding the behavior of a function by looking at the restriction to the central fiber of a related $1$-form is very similar to Coleman's method of effective Chabauty in the bad reduction case as explored by Lorenzini-Tucker \cite{LT} and McCallum-Poonen \cite{MP}.   Since Abelian varieties have a theory of toric degenerations \cite{Mumford}, it may be possible to use our method to get bounds on the number of rational points that specialize to certain components in a degeneration of a curve in its Jacobian.

We give an outline of the paper.  Background on specialization of linear systems is given in section 2.  The notation used in the statement of Theorem \ref{maintheorem} is defined in section 3.  Section 4 shows that Theorem \ref{maintheorem} implies a generalization of Speyer's well-spacedness condition while section 5 shows that our obstruction is new by proving Theorem \ref{weakwellspaced} and giving a graph that does not lift to a curve of genus $3$ by our obstruction but which is not obstructed by any other known condition.  We assemble background on toric schemes, log structures, and tropicalization in section 6.  Section 7 defines the vanishing functions of sections of line bundles while section 8 studies the properties of the vanishing function of log differentials of generalized units.  We apply these methods to prove Theorem \ref{maintheorem} in section 9.

We'd like to acknowledge Matt Baker, Tristram Bogart, Erwan Brugall\'{e}, Mark Gross, David Helm, Gregg Musiker, Sam Payne, Bernd Siebert, Frank Sottile, and David Speyer for valuable discussions. We'd also like to thank Johannes Rau and the anonymous referee for a careful reading and helpful comments.  The author was  partially supported by NSF grant DMS-0441170 administered by the Mathematical Sciences Research Institute (MSRI) while the author was in residence at MSRI during the Tropical Geometry program, Fall 2009.  We would like to thank the organizers of the program and MSRI staff for making such a program possible.

\section{Specialization from Curves to Graphs}

We review some results on specialization of linear systems from curves to graphs due to Baker \cite{B}.  Our approach is a slight enlargement of his methods in that we allow curves with marked points.
A semistable family of curves $\cc$ over $\Spec \O$ is a family of curves such that the generic fiber $\cc_\K$ is smooth and the central fiber $\cc_0$ is reduced with only ordinary double points as singularities.  By blowing up the singular points, one can ensure that $\cc$ is regular.  A node of $\cc_0$ is formal locally parameterized by $\O[x,y]/(xy-t^l)$.  The family is regular near the node if and only if $l=1$.  A marked semistable family of curves is a family $\cc$ together with disjoint sections, $\sigma_1,\dots,\sigma_n:\Spec\O\rightarrow \cc^{\sm}$ valued in the smooth locus.

\begin{definition} The {\em dual graph} $\Sigma$ of a marked semistable curve $\cc$ is a graph with bounded and labelled unbounded edges (called leaves) whose vertices correspond to  components of the normalization $\pi:\widetilde{\cc}_0\rightarrow\cc_0$, whose edges correspond to nodes, and whose leaves correspond to marked points.  To each vertex $v$ is associated the genus $g(v)$ of the corresponding  component $C_v$ of $\widetilde{\cc}_0$.
\end{definition}

The genus of the dual graph is defined to be
\[g(\Sigma)=h^1(\Sigma)+\sum_v g(v).\]
In general, $g(\Sigma)$ is equal to the genus of $\cc_\K$.
$\cc$ is said to be {\em maximally degenerate} if all components of $\cc_0$ are rational.  In this case, $g(\cc_\K)=h^1(\Sigma)$.

Let $\cc$ be a semistable regular family of curves over $\O$ with dual graph $\Sigma$.  Let $V(\Sigma)$, $E(\Sigma)^\bullet$, $E(\Sigma)^\circ$ be the vertices, bounded edges, and leaves of $\Sigma$.  For $v\in V(\Sigma)$, let $C_v$ be the irreducible component of $\widetilde{\cc}_0$ corresponding to $v$.  For an edge $e\in E(\Sigma)^\bullet$ between $v_1,v_2$, let $p_e$ be the corresponding node in $\cc_0$ between $C_{v_1}$ and $C_{v_2}$.  For an unbounded edge $e'\in E(\Sigma)^\circ$, let $\sigma_{e'}$ be the corresponding marked point.   Write $p_{e'}=\sigma_{e'}(\Spec \k)$.  We fix 
homeomorphisms of the bounded edges with $[0,1]$ and of the unbounded edges with 
$[0,\infty)$.  The graph is given the induced metric.  We will often find it necessary to replace the generic fiber $\cc_\K$ by its base-change $\cc_\K\times_\K \K[t^{\frac{1}{l}}]$.  This is the generic fiber of the family 
$\cc\times_\O \O[t^{\frac{1}{l}}]$.  This family is no longer regular since the nodes are formal locally isomorphic to $\O[u][x,y]/(xy-u^l)$ where $u=t^{\frac{1}{l}}$.  The family becomes regular once we blow-up each node $l-1$ times giving a new model $\cc^l$ with a chain of $l-1$ rational curves in place of each node.  This has the effect of rescaling each bounded edge by a factor of $l$ and then subdividing it into $l$ equal edges each homeomorphic to $[0,1]$ producing a new dual graph $\Sigma_l$.  By convention, given an unbounded edge $e$ with homeomorphism $j:e\rightarrow [0,\infty)$, we define the corresponding homeomorphism in $\Sigma_l$, $j_l:e\rightarrow [0,\infty)$ by $j_l(s)=lj(s)$.

A divisor on $\Sigma$ is an element of the free abelian group on $V(\Sigma)$.  We write a divisor as $D=\sum_{v\in V(\Sigma)} a_v(v)$.  The group of all divisors is denoted by $\Div(\Sigma)$.  We say a divisor $D$ is non-negative  and write $D\geq 0$ if $a_v\geq 0$ for all $v\in V(\Sigma)$.  We write $D\geq D'$ if $D-D'\geq 0$.   The {\em canonical divisor} on $\Sigma$ is 
\[K_\Sigma=\sum_{v\in V(\Sigma)} (\deg(v)+2g(v)-2)(v).\]
In the maximally degenerate case, this is equal to the canonical divisor of \cite{BN}.
A {\em piecewise-linear function} $\varphi$ on $\Sigma$ is a function $\varphi:\Sigma\rightarrow \R$ that is linear on each edge.  Let $\varphi$ be a piecewise linear function on $\Sigma$ with integer slopes.  For $e\in E(\Sigma)$ and $v\in e$, let $s_{\varphi}(v,e)$ be the slope of $\varphi$ at $v$ along $e$ (oriented away from $v$).  The Laplacian of $\varphi$, $\Delta(\varphi)$ is the divisor on $\Sigma$ given by
\[\Delta(\varphi)=-\sum_{v\in V(\Sigma)} \sum_{e\ni v} s_\varphi(v,e)(v).\]

For $\Sigma$, the dual graph of $\cc$, the specialization map $\rho:\Div(\cc)\rightarrow\Div(\Sigma)$ is defined by,
for  $\cd\in\Div(\cc)$,
\[\rho(\cd)=\sum_{v\in\Gamma}\deg(\pi^*\O(\cd)|_{C_v})(v).\]

\begin{definition} Let $\Lambda$ be a divisor on $\Sigma$.  If $\varphi$ is a piecewise-linear function on $\Sigma$ with $\Delta(\varphi)+\Lambda\geq 0$, we write $\varphi\in L(\Lambda)$ to denote that $\varphi$ is in the linear system given by $\Lambda$.
\end{definition}

\section{Lifting Condition}

In this section, we define the terms in the statement of Theorem \ref{maintheorem} and draw some combinatorial consequences.
Let $f:C^*\rightarrow (\K^*)^n$ be a morphism of a smooth curve with regular semistable completion $f:\cc\rightarrow\cp$ and parameterized tropicalization (in the sense of Section \ref{s:tslst}) $\Trop(f):\Sigma\rightarrow N_\R$ where $\Sigma$ is the  dual graph of $\cc$.  For $\Gamma$ a subgraph of $\Sigma$, let $(\cc_\Gamma)_0$ be the subcurve of $\cc_0$ that is the union of components corresponding to vertices of $\Gamma$, that is,
\[(\cc_\Gamma)_0=\bigcup_{v\in\Gamma} C_v.\]  For $e\in E(\Sigma)$ and $m\in M$, let $m\cdot e$ denote the inner product $\<m,\Trop(f)|_e(1)-\Trop(f)|_e(0)\>$.

We define a tropical homomorphism.  For $a,b,c\in\R$, we write $a\oplus b\oplus c=0$ if the minimum of $\{a,b,c\}$ is achieved at least twice.  For $(-\infty,\infty]$-valued functions, $f,g,h$, we write $f\oplus g\oplus h=0$ if $f(x)\oplus g(x)\oplus h(x)=0$ for all $x$.

\begin{definition} Let $L$ be a set of $(-\infty,\infty]$-valued functions on a graph $\Sigma$.  For $G$, an abelian group, a {\em tropical homomorphism} of $G$ to $L$ is a function $a: G\rightarrow L$ such that 
\begin{enumerate}
\item $a(e)=\infty$,
\item For $g_1,g_2\in G$, $a(g_1)\oplus a(g_2) \oplus a(g_1+g_2)=0.$
\end{enumerate}
\end{definition}

For a subgraph $\Gamma\subset\Sigma$, let $\partial\Gamma$ denote the boundary of $\Gamma$ considered as a closed subset of $\Sigma$.  

\begin{definition} Let $\Gamma\subseteq\Gamma'$ be subgraphs of $\Sigma$ such that $\Gamma$ is bounded, connected, and contained in the interior of $\Gamma'$.   Let $\varphi$ be a piecewise-linear function on $\Sigma$ that has integer slopes on edges.  Set
\begin{eqnarray*}
h&=&\min_{v\in\Gamma} \varphi(v)\\
E^\partial_v&=&\{e\in E(\Gamma')\setminus E(\Gamma)| e\ni v\}.
\end{eqnarray*}
Define the effective divisor $D_\varphi$ on $(\cc_\Gamma)_0$ by 
\[D_\varphi=\sum_{v\in \partial\Gamma,\varphi(v)=h} \left(\sum_{e \in E_v^\partial,s_\varphi(v,e)<0} -s_\varphi(v,e)(p_e)\right).\]
We say {\em $\varphi$ is $\cc_0$-ample on $\Gamma$ in $\Gamma'$} if
the invertible sheaf $\O_{(\cc_\Gamma)_0}(D_\varphi)$ has a section which is non-constant exactly on components $C_v$ for vertices $v$ with $\varphi(v)=h$.
In other words, there is a meromorphic function $f$ on $(C_\Gamma)_0$ with $(f)+D_\varphi\geq 0$ that is constant exactly on the components corresponding to $v\in\Gamma$ with $\varphi(v)>h$. 
\end{definition}

Under certain circumstances, $\cc_0$-ampless is a combinatorial condition.  
\begin{lemma}
Let $\Gamma$ be a graph all of whose vertices have degree $2$ or $3$.  If all components of $\cc_0$ are rational, then $\cc_0$-ampleness is determined by $(\Sigma,\Gamma,\Gamma',\varphi)$.  
\end{lemma}

\begin{proof} 
In this case, $(\cc_\Gamma)_0$ is completely determined by its dual graph $\Gamma$.   Let $\pi:(\widetilde{\cc}_\Gamma)_0\rightarrow (\cc_\Gamma)_0$ be the normalization map.  Then the normalization exact sequence gives
\[\xymatrix{
0\ar[r]&H^0((\cc_\Gamma)_0,\O(D_\varphi))\ar[r]&\bigoplus_{v\in V(\Gamma)} H^0(C_v,\O(D_\varphi|_{C_v}))\ar[r]^>>>>>r&\bigoplus_{e\in E(\Gamma)^\bullet}\O(D_\varphi)|_{p_e}\\
\ar[r]&H^1((\cc_\Gamma)_0,\O(D_\varphi))\ar[r]&\bigoplus_{v\in V(\Gamma)} H^1(C_v,\O(D_\varphi|_{C_v}))\ar[r]&0.
}\]
The existence of the desired section is exactly the condition that $\ker(r)$ intersects
\[\bigoplus_{v\in V(\Gamma)|\varphi(v)=h} (H^0(C_v,\O(D_\varphi|_{C_v}))\setminus \C_{C_v})\oplus \bigoplus_{v\in V(\Gamma)|\varphi(v)>h} \C_{C_v}\]
 where $\C_{C_v}$ denotes the constant functions on $C_v$ viewed as elements of $H^0(C_v,\O(D_\varphi|_{C_v}))$.
\end{proof}

\begin{lemma}
Suppose $(\cc_\Gamma)_0$ is a curve of arithmetic genus $0$.  $\varphi$ is  $\cc_0$-ample on $\Gamma\subset\Gamma'$ if and only if  and $\deg(D_\varphi|_{C_v})\geq 1$ for all $v\in V(\Gamma)$ with $\varphi(v)=h$.  
\end{lemma}

\begin{proof}
Suppose $\deg(D_\varphi|_{C_v})=0$ for some $v\in V(\Gamma)$ with $\varphi(v)=h$.  Then the only sections of $\O_{C_v}(D_\varphi|_{C_v})$ are constants.  This contradicts $\cc_0$-ampleness.

Suppose $\deg(D_\varphi|_{C_v})\geq 1$ for all vertices with $\varphi(v)=h$.  Pick a non-constant section of $\O_{C_v}(D_\varphi)$ for each $v$ with $\varphi(v)=h$.  Such sections can be chosen to be non-zero at points $p_e$ corresponding to edges of $\Gamma$.  Pick a non-zero constant section for each $v$ with $\varphi(v)>h$.  Since $\Gamma$ is a tree, one can replace the sections on each component with non-zero constant multiples to get agreement across nodes.  \end{proof}

In general, $\cc_0$-ampleness is not a combinatorial condition and may depend on the curve $\cc_0$.   Below, we give one consequence of $\cc_0$-ampleness which applies generally and which we will use to obtain necessary conditions for lifting.

\begin{lemma}  \label{cycledeg}  Let $\Gamma\subset\Sigma$ be a $2$-vertex connected graph with no $1$-valent vertices. If  $\varphi$ is $\cc_0$-ample on $\Gamma\subset\Gamma'$ 
 then $\deg(D_\varphi)\geq 2$ and $\deg(D_\varphi|_{C_v})\geq 1$ for all $v\in V(\Gamma)$ with $\varphi(v)=h$.  
\end{lemma}

\begin{proof}
If $\deg(D_\varphi)=0$ then there is no global non-constant section of $\O(D_\varphi)$ so we may suppose $\deg(D_\varphi)=1$.  Then $D_\varphi$ is supported on some component $C_v$.  This implies that $f$ can be interpreted as a degree $1$ rational function on $C_v$ and is constant on $C'=(\cc_\Gamma)_0\setminus C_v$ which is connected.  Since $C'$ meets $C_v$ in at least two points, $f$ must attain the same value several times on $C_v$ which is impossible.

The condition on $\deg(D_\varphi|_{C_v})$ follows as in the above lemma.
\end{proof}

The above lemma can be thought of as the analog of the statement that a non-constant rational function on a curve with $g(C)\geq 1$ has at least two poles (counted with multiplicity).  If $\Gamma$ is a cycle of rational curves, the converse is true.

\begin{lemma}  Let $\Gamma\subset\Sigma$ be a cycle with $g(v)=0$ for all $v\in V(\Gamma)$.  Then  $\varphi$ is $\cc_0$-ample on $\Gamma\subset\Gamma'$ if and only if $\deg(D_\varphi)\geq 2$ and $\deg(D_\varphi|_{C_v})\geq 1$ for all $v\in V(\Gamma)$ with $\varphi(v)=h$.  
\end{lemma}

\begin{proof}
Necessity follows from the lemma above.

Suppose $\deg(D_\varphi)\geq 2$.  Let $v_0,\dots,v_k$ be the vertices in the cycle ordered 
cyclically.  Let $e_i$ be the edge from $v_i$ to $v_{i+1}$.  Let $p_i$ be the corresponding node. 
Suppose $D_\varphi$ is supported on $C_{v_{i_l}}$, for $i_1<i_2<\dots<i_d$.  Now, pick non-constant sections $f_{i_l}$ of $\O(D_\varphi)$ on $C_{v_{i_l}}$ such that 
$f_{i_l}(p_{i_l})=f_{i_{l+1}}(p_{i_{l+1}-1})$.  This is possible because  $\deg(D_\varphi)\geq 2$.  Now, choose $f_i$ to be a constant on those $C_{v_i}$ disjoint from the support of $D_\varphi$.  These $f_i$'s can be chosen to agree on the nodes $p_i$ and therefore give the desired section.
\end{proof}

The conditions from Theorem \ref{maintheorem} can be translated to combinatorial constraints on $\varphi_m$ which we will use in later sections.  
\begin{lemma} \label{slopesums} Suppose that $g(v)=0$ for all vertices of $\Sigma$ and $\varphi$ satisfies conditions (1)-(3) of Theorem \ref{maintheorem}
For any vertex $v$,
\[\sum_{e\ni v|m\cdot e=0} s_{\varphi_m}(v,e)\leq \deg(v)-2.\]
If $v\in V(\Sigma)$ is a vertex with an adjacent edge $e_1$ such that  $m\cdot e_1\neq 0$, then
\begin{enumerate}
\item \label{l:e1} $\varphi_m(v)=0$,

\item \label{l:e2} the slope satisfies $1\leq s_{\varphi_m}(v,e)$ for all $e$ with $m\cdot e=0$,

\item \label{l:e3} if $v$ is trivalent, then there is at most  one adjacent edge $e$ with $m\cdot e=0$.  For that edge, we have $s_{\varphi_m}(v,e)=1$.
\end{enumerate}
\end{lemma}

\begin{proof}
Since $K_\Sigma(v)=\deg(v)-2$ and $\varphi_m$ vanishes on edges not orthogonal to $m$, the inequality on slopes is a consequence of $\varphi_m\in L(K_\Sigma)$.

Condition (\ref{l:e1}) follows from $\varphi_m$ vanishing on edges not orthogonal to $m$.  Condition (\ref{l:e2}) is just the fact that $\varphi_m$ is non-negative and must have integral slopes on edges.  Condition (\ref{l:e3}) follows from balancing and the above inequality on slopes.
\end{proof}

We refer to a chain of edges connected by $2$-valent vertices as {\em smooth segments}.  By the above lemma, slopes of $\varphi_m$ are always non-increasing along smooth segments, hence $\varphi_m$ is concave there.

We now state the definition of tropical parameterization used in Corollary \ref{c:istrop}.  Below, a {\em weighted graph} is an abstract graph with non-negative integer weights on the edges.  Each edge is given length $1$.  A {\em balanced integral graph} $(\Sigma',\mu')$ is an immersed graph in $\R^n$ whose vertices have integer coordinates and which satisfies the balancing condition.  Each edge is given a length equal to its lattice length.  A {\em genus-marking} of a graph $\Sigma$ is a function $g:V(\Sigma)\rightarrow \Z_{\geq 0}$.

\begin{definition} \label{d:tparameterization} A {\em tropical parameterization}  of a 
balanced weighted integral graph $(\Sigma',\mu')$ immersed in $\R^n$ is a genus-marked weighted graph $(\Sigma,\mu,g)$ together with a map $p:\Sigma\rightarrow\Sigma'$ such that 
\begin{enumerate}
\item \label{par:length} For each edge $e\in E(\Sigma)$, $p|_e$ acts as dilation by a factor $\mu(e)$.

\item \label{par:bal} If each edge $e\in E(\Sigma)$ not contracted by $p$ is assigned the primitive integer direction $w(p(e))$ of $p(e)$, then $\Sigma$ is a balanced graph in the following sense: for any $v\in V(\Sigma)$,
\[\sum_{e\ni v} \mu(e)w(p(e))=0.\]

\item \label{par:mult} For any edge $e'\in E(\Sigma')$, we have 
\[\sum_{e \in p^{-1}(e')} \mu(e)=\mu(e'),\]

\item \label{par:semistable} if $v\in V(\Sigma)$ is a vertex all of whose edges are contracted by $p$ then the degree of $v$ is at least $2$.
\end{enumerate}
\end{definition}


The genus of $(\Sigma,\mu,g)$ is 
\[g(\Sigma)=h^1(\Sigma)+\sum_v g(v).\]

Note that if $(\Sigma',\mu')$ has only vertices of degree $2$ or $3$ and has all multiplicities equal to $1$, the only possible semistable tropical parameterization with all vertices of genus $0$ is the identity map.  We may study tropical parameterizations more generally.  We say that a vertex $(\Sigma',\mu')$ is {\em indecomposable} if its star cannot be written as the union (with multiplicities) of two proper balanced subgraphs.   If a vertex is indecomposable, it is impossible to insert edges contracted by $p$ at it in the parameterization.  Therefore, if each vertex is indecomposable, there are finitely many tropical parameterizations of $\Sigma'$  of a fixed genus.  These parameterizations correspond to different choices of pre-images of edges, different  combinatorial types of $\Sigma$, and different genus-markings.  If, in addition, all multiplicities of $(\Sigma',\mu')$ are $1$ then the only tropical parameterization with all vertices of genus $0$ is the identity map since the pre-image under $p$ of any edge is a single edge.  

\section{Comparison to Known Obstructions}

In this section, we relate $\cc_0$-ampleness of a function $\varphi$ on $\Gamma$ in $\Gamma'$ in a genus-marked graph $\Sigma$ to the necessity of Speyer's well-spacedness condition for genus $1$ curves to lift and its higher genus generalization by Nishnou \cite{N}.  We make the following assumptions:
\begin{enumerate}
\item All vertices of $\Sigma$ have degree $2$ or $3$,
\item \label{a:forest} $\Sigma\setminus\Gamma$ is a forest, and
\item \label{a:vc} $\Gamma$ is a bounded, $2$-vertex connected subgraph of $\Gamma'\setminus\partial\Gamma'$ with no $1$-valent vertices.
\item \label{a:genus} the genus $g(v)$ of every vertex of $\Sigma$ is $0$.
\end{enumerate}
Assumptions (\ref{a:forest}) and (\ref{a:vc}) are related to Nishinou's one-bouquet condition \cite{N} which is required for his necessary and sufficient generalization of Speyer's well-spacedness condition to apply.
Assumption (\ref{a:genus}) ensures that 
\[K_\Sigma=\sum_{v\in V(\Sigma)} (\deg(v)-2)(v).\]

The following condition is equivalent to Speyer's well-spacedness condition \cite{Spe05} in the genus $1$ case and generalizes it in higher genera:
\begin{proposition} \label{wellspaced} 
Let $\Trop(f):\Sigma\rightarrow\R^n$ be the parameterized tropicalization of a map of a smooth curve.  Let $m\in M$, $c\in\R$, and $H=\{x|x\cdot m=c\}$ and $\Gamma'=\Trop(f)^{-1}(H)$.
If $\Gamma$ is as above then the minimum of $\dist(w,\Gamma)$ for $w\in\partial\Gamma'$ must be achieved for at least two values of $w$.
\end{proposition}

We first go through an example to show how to derive well-spacedness in genus $1$ from our condition.
Consider a tropical elliptic curve $\Trop(f)(\Sigma)$ in $\R^3$ such that $\Gamma'$ looks like the following graph: 
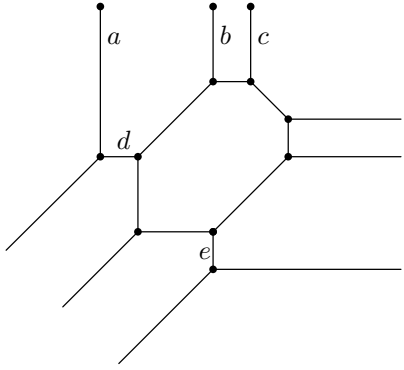
\begin{center}\begin{texdraw}
       \drawdim cm  \relunitscale 0.5
       \linewd 0.03
        \move(-.5 -.5) \lvec (2 2)
        \lvec(7 2)
        \move(2 2) \fcir f:0 r:0.1
        \lvec(2 3) \fcir f:0 r:0.1
        \lvec(4 5) \fcir f:0 r:0.1
        \lvec(4 6)\fcir f:0 r:0.1
        \lvec(3 7)\fcir f:0 r:0.1
        \lvec(2 7)\fcir f:0 r:0.1
        \lvec(0 5)\fcir f:0 r:0.1
        \lvec(0 3)\fcir f:0 r:0.1
        \lvec(2 3)\fcir f:0 r:0.1
       \move(0 3)
       \lvec(-2 1)
       \move(0 5)
       \lvec(-1 5)
       \lvec(-1 9)\fcir f:0 r:0.1
       \move(-1 5)\fcir f:0 r:0.1
       \lvec(-3.5 2.5)
       \move(3 7)
       \lvec(3 9)\fcir f:0 r:0.1
       \move(2 7)
       \lvec(2 9)\fcir f:0 r:0.1
       \move(4 5)\lvec(7 5)
       \move(4 6)\lvec(7 6)
       \htext(-.8 8){\footnotesize $a$}
       \htext(2.2 8){\footnotesize $b$}
       \htext(3.2 8){\footnotesize $c$}
       \htext(-.55 5.2){\footnotesize $d$}
       \htext(1.65 2.25){\footnotesize $e$}
\end{texdraw}
\end{center}
Here edges may be subdivided,  those edges not terminating in a vertex are taken to be unbounded, and each edge is given multiplicity $1$.  Take $\Gamma$ to be the cycle and orient the edges not in the cycle so that they point towards $\Gamma$.  By Lemma \ref{cycledeg}, the divisor $D_{\varphi_m}$ must have degree at least $2$.  The degree of $D_{\varphi_m}$ is the sum of the positive slopes of edges coming into $\Gamma$ at points where $\varphi_m$ is minimized.   $\varphi_m$ is concave and non-negative  along unbounded edges so they cannot contribute to the divisor $D_{\varphi_m}$ on the cycle.  The slope along edges $a,b,c$ is at most $1$ and is non-increasing along them since $\Delta(\varphi_m)\geq 0$ on the interior of those edges.  By Lemma \ref{slopesums}, the slope on edge $e$ must be non-positive.   Similarly, the slope on edge $d$ must be less than or equal to the smallest slope on $a$.  Therefore, the only positive slopes entering $\Gamma$ must come from edges $d$,$b$,$c$.  At the points where those edges intersect $\Gamma$, $\varphi_m$ must be less than or equal to the distance from those points to $\partial\Gamma'$.  Equality is achieved if and only if the slope is $1$ at those points.  It follows that since the value of $\varphi_m$ on $\Gamma$ must be minimized at two of those points where the slope is $1$, the minimum of $\{|a|+|d|,|b|,|c|\}$ must be achieved at least twice.

Now, we consider the general situation.  Let $v\in\partial\Gamma$ and let $T_v$ be the unique tree in $\Sigma\setminus(\Gamma\setminus\partial\Gamma)$ containing $v$.  Direct the edges of $T_v$ so that they point towards $v$.  Write $T_v\cap\partial\Gamma'=\{w_1,\dots,w_k\}$.  Let $\gamma_1,\dots,\gamma_k$ be the paths from $w_1,\dots,w_k$ to $v$.  Let $l_i$ be the length of $\gamma_i$, $l(v)=\min({l_i})$, $S_v=\{i|l_i=l(v)\}$.  

We will need to describe a pruning operation on the graph $\Sigma$ that will remove $T_v$

\begin{lemma} \label{pruning}  Suppose
\begin{enumerate}
\item $\Delta(\varphi_m)+K_\Sigma\geq 0$, and 
\item $\varphi_m$ is $\cc_0$-ample with respect to $\Gamma$ in $\Gamma'$.
\end{enumerate}
Let $v\in\partial\Gamma$, and let  $e$ be the unique edge of $T_v$ adjacent to $v$.
If $s(v,e)>0$ then, for $\Sigma'=\Sigma\setminus (T_v\setminus\{v\})$,
$\varphi_m|_{\Sigma'}$ satisfies:
\begin{enumerate}
\item $\Delta(\varphi_m|_{\Sigma'})+K_{\Sigma'}\geq 0$, and 
\item $\varphi_m|_{\Sigma'}$ is $\cc_0$-ample with respect to $\Gamma$ in $\Gamma'\setminus (T_v\setminus\{v\})$.
\end{enumerate}

\end{lemma}

\begin{proof}
It is clear that $\varphi_m|_{\Sigma'}$ satisfies the $\cc_0$-ampleness condition since the edge $e$ does not contribute to $D_{\varphi_m}$.

If $s_{\varphi_m}(v,e)>0$ then $\Delta(\varphi_m |_{\Sigma'})(v)\geq \Delta(\varphi_m)(v)+1$ while $K_\Sigma(v)=K_{\Sigma'}(v)-1$.  
Therefore, 
\[\Delta(\varphi_m|_{\Sigma'})(v)+K_{\Sigma'}(v)\geq \Delta(\varphi_m)(v)+K_\Sigma(v)\geq 0.\]
\end{proof}

The following lemma constrains the value of $\varphi_m$ at points of $\partial\Gamma$ that could possibly contribute to $D_{\varphi_m}$.

\begin{lemma} Suppose $\varphi_m$ is a non-negative piecewise linear function on $\Sigma$ satisfying 
$\Delta(\varphi_m)+K_{\Sigma}\geq 0$ with slope $1$ near $\partial\Gamma'$ and 
$\varphi_m|_{\partial\Gamma'}=0$.  Suppose also that $\varphi_m$ never has slope $0$ on 
$\Gamma'$.  Let $v\in\partial\Gamma$ and $e$ be the edge of 
$\Gamma'\setminus\Gamma$ adjacent to $v$.  If $s_{\varphi_m}(v,e)<0$ then $\varphi_m(v)\geq l(v)$.  If, in addition, $\#S_v=1$, then $s_{\varphi_m}(v,e)=-1$ and $\varphi_m(v)=l(v)$. \label{decpaths}
\end{lemma}

\begin{proof} 
Write $e=vv_1$.  Suppose $v_1\not\in\partial\Gamma'$. The inequality $\Delta(\varphi_m)+K_\Sigma\geq 0$ implies
\[\sum_{e'} s_{\varphi_m}(v_1,e')\leq\deg(v_1)-2,\]
and so there is an adjacent edge $e_1$ with $e_1\neq e$, $s_{\varphi_m}(v_1,e_1)<0$.  If $e_1=v_1v_2$ and $v_2\not\in\partial\Gamma'$ then there is an edge $e_2\neq e_1$ adjacent to $v_2$ with $s_{\varphi_m}(v_2,e_2)<0$.  Applying this argument repeatedly, we can find a path in $\Gamma'\setminus\Gamma$ from $v$ to some $w_i$ with all slopes negative.  From $\varphi_m(w_i)=0$, we obtain $\varphi_m(v)\geq l_i\geq l(v)$.

Now suppose $\#S_v=1$.  Without loss of generality, suppose $S_v=\{1\}$.   We claim that $\varphi_m$ is linear with slope $1$ along $\gamma_1$.  Let $v'$ be the first non-smooth
vertex along $\gamma_1$.  Since slopes are non-increasing along $\gamma_1$ from $w_1$ to $v'$, $\varphi_m(v')\leq\dist(v',w_1)$. Let $T'$ be a component of 
$T_v\setminus \{v'\}$ not containing $w_1$.  Let $S=\{i|w_i\in T'\}$.  Let $e'$ be the edge in $T'$ adjacent to $v'$.  If $s(v',e')<0$  then by the first part of this lemma, 
\[\varphi_m(v')\geq \min_{i\in S}(\dist(v',w_i))>\dist(v',w_1)\geq\varphi_m(v').\]
This contradiction proves that $s(v,e')>0$, and so by Lemma \ref{pruning}, we may remove $T'$ from $\Gamma'$.  By continuing this argument, we may eliminate all such trees $T'$.  Therefore, we may suppose that $v'$ is a smooth vertex of $\Gamma'$.  Continuing this argument over $\gamma_1$, we may suppose that $\gamma_1$ is a smooth path.  The slope of $\varphi_m$ must be non-increasing along $\gamma_1$.  Since it begins 
with slope $1$ and ends with a positive slope, $\varphi_m$ must be linear on $\gamma_1$.  It follows that $\varphi_m(v)=l(v)$.
\end{proof}

The proof of Proposition \ref{wellspaced} is completed by the following:

\begin{lemma} If $\Gamma$ is a bounded $2$-vertex connected graph with no $1$-valent vertices and $\varphi_m$ is $\cc_0$-ample on $\Gamma\subset\Gamma'\setminus\partial\Gamma'$ then 
the minimum $\min_{w\in\partial\Gamma'} \dist(w,\Gamma)$ must be achieved at least twice.
\end{lemma}

\begin{proof}
By Lemma \ref{cycledeg}, we must have $\deg(D_{\varphi_m})\geq 2$.   Let 
$l=\min_{v\in\partial\Gamma} l(v)$ and $F=\{v\in\partial\Gamma|l(v)=l\}$.  If $\#F\geq 2$ then we 
are done, so we may suppose $F=\{v\}$.  If $\#S_v\geq 2$, we are also done, so we may suppose 
$\#S_v=1$.  By the previous lemma, $s(v,e)=-1$ and $\varphi_m(v)=l(v)=l$.  For $w\in\partial\Gamma\setminus\{v\}$ and $e\in E(\Gamma')\setminus E(\Gamma)$ adjacent to $w$,
we have either $s_{\varphi_m}(w,e)>0$ or
\[\varphi_m(w)\geq l(w)>l=\varphi_m(v).\]
In either case, we can conclude that $w$ does not contribute to $D_{\varphi_m}$.  Consequently, $\deg(D_{\varphi_m})=1$ which is a contradiction.
\end{proof}

\section{New Obstructions}
In this section, we prove Theorem \ref{weakwellspaced} and give an example of a graph whose lifting is obstructed by Theorem \ref{maintheorem} but not by any other known lifting conditions.  We begin with the proof of Theorem \ref{weakwellspaced}:

\begin{proof}
Suppose that $\partial\Gamma'$ is a single trivalent vertex $v$.  By balancing, only one edge $e$ at this vertex can map into $H$.  Now let $\Gamma$ be a cycle in $\Gamma'$ such that the following expression is minimized:
\[h_\Gamma=\min_{v'\in V(\Gamma)} \varphi(v').\]
By Theorem \ref{maintheorem} and Lemma \ref{cycledeg}, the divisor $D_{\varphi_m}$ on $\Gamma$ satisfies $\deg(D_{\varphi_m})\geq 2$.  We first observe that there is at most one path from points of $D_{\varphi_m}$ to $v$ on which $\varphi_m$ is decreasing.  If there was more than one path, they would have non-empty intersection containing $e$.  Then one could construct a new cycle $\Gamma_2$ from segments of $\Gamma$ and the paths for which $h_{\Gamma_2}<h_\Gamma$.

By applying the inequality
\[\sum_{e'} s_{\varphi_m}(v',e')\leq\deg(v')-2\]
as in the proof of Lemma \ref{decpaths}, we can construct paths from points of $D_{\varphi_m}$ to $v$ on which $\varphi_m$ is decreasing.  This implies that $D_{\varphi_m}$ must consist of a single point with multiplicity at least $2$.  The inequality above together with the uniqueness of the path shows that $\varphi_m$ must decrease with slope at most $-2$ along the path we construct from $D_{\varphi_m}$ to $v$.  Now, since $\varphi_m$ is equal to $0$ on the edges of $v$ different from $e$, the above inequality gives $s_{\varphi_m}(v,e)\leq \deg(v)-2=1$.  This contradiction proves the lemma.
\end{proof}

Now, we give a tropical curve that does not satisfy the conditions of Theorem \ref{maintheorem} but to which Propositions \ref{wellspaced} and \ref{weakwellspaced} do not apply.
We claim that there is a balanced parameterized graph $h:\Sigma\rightarrow\R^3$ and that there is a rational hyperplane $H$ such that $h^{-1}(H)$ is the following graph where all edges have multiplicity $1$:
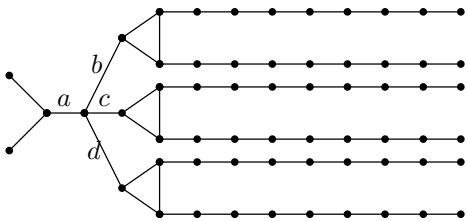
\begin{center}
   \begin{texdraw}
       \drawdim cm  \relunitscale 0.5
       \linewd 0.03
        \move(-1 -1) \fcir f:0 r:0.1 \lvec (0 0)
         \fcir f:0 r:0.1
        \lvec(1 0) \fcir f:0 r:0.1
        \move(-1 1) \fcir f:0 r:0.1 \lvec(0 0) \fcir f:0 r:0.1
        \move(1 0) \fcir f:0 r:0.1 \lvec(2 2) \fcir f:0 r:0.1
         \fcir f:0 r:0.1
        \lvec(3 2.7)\fcir f:0 r:0.1
        \lvec(4 2.7)\fcir f:0 r:0.1
        \lvec(5 2.7)\fcir f:0 r:0.1
        \lvec(6 2.7)\fcir f:0 r:0.1
        \lvec(7 2.7)\fcir f:0 r:0.1
        \lvec(8 2.7)\fcir f:0 r:0.1
        \lvec(9 2.7)\fcir f:0 r:0.1
        \lvec(10 2.7)\fcir f:0 r:0.1
        \lvec(11 2.7)\fcir f:0 r:0.1
        \move(3 2.7)\fcir f:0 r:0.1
        \lvec(3 1.3)\fcir f:0 r:0.1
        \lvec(4 1.3)\fcir f:0 r:0.1
        \lvec(5 1.3)\fcir f:0 r:0.1
        \lvec(6 1.3)\fcir f:0 r:0.1
        \lvec(7 1.3)\fcir f:0 r:0.1
        \lvec(8 1.3)\fcir f:0 r:0.1
        \lvec(9 1.3)\fcir f:0 r:0.1
        \lvec(10 1.3)\fcir f:0 r:0.1
        \lvec(11 1.3)\fcir f:0 r:0.1
        \move(3 1.3)\fcir f:0 r:0.1
        \lvec(2 2)\fcir f:0 r:0.1
        \move(1 0)\fcir f:0 r:0.1
        \lvec(2 -2)\fcir f:0 r:0.1
        \lvec(3 -2.7)\fcir f:0 r:0.1
        \lvec(4 -2.7)\fcir f:0 r:0.1
        \lvec(5 -2.7)\fcir f:0 r:0.1
        \lvec(6 -2.7)\fcir f:0 r:0.1
        \lvec(7 -2.7)\fcir f:0 r:0.1
        \lvec(8 -2.7)\fcir f:0 r:0.1
        \lvec(9 -2.7)\fcir f:0 r:0.1
        \lvec(10 -2.7)\fcir f:0 r:0.1
        \lvec(11 -2.7)\fcir f:0 r:0.1
        \move(3 -2.7)\fcir f:0 r:0.1
        \lvec(3 -1.3)\fcir f:0 r:0.1
        \lvec(4 -1.3)\fcir f:0 r:0.1
        \lvec(5 -1.3)\fcir f:0 r:0.1
        \lvec(6 -1.3)\fcir f:0 r:0.1
        \lvec(7 -1.3)\fcir f:0 r:0.1
        \lvec(8 -1.3)\fcir f:0 r:0.1
        \lvec(9 -1.3)\fcir f:0 r:0.1
        \lvec(10 -1.3)\fcir f:0 r:0.1
        \lvec(11 -1.3)\fcir f:0 r:0.1
        \move(3 -1.3)\fcir f:0 r:0.1
        \lvec(2 -2)
        \move(1 0)
        \lvec(2 0)
         \lvec(3 -.7)\fcir f:0 r:0.1
        \lvec(4 -.7)\fcir f:0 r:0.1
        \lvec(5 -.7)\fcir f:0 r:0.1
        \lvec(6 -.7)\fcir f:0 r:0.1
        \lvec(7 -.7)\fcir f:0 r:0.1
        \lvec(8 -.7)\fcir f:0 r:0.1
        \lvec(9 -.7)\fcir f:0 r:0.1
        \lvec(10 -.7)\fcir f:0 r:0.1
        \lvec(11 -.7)\fcir f:0 r:0.1
        \move(3 -.7)\fcir f:0 r:0.1
        \lvec(3 .7)\fcir f:0 r:0.1
        \lvec(4 .7)\fcir f:0 r:0.1
        \lvec(5 .7)\fcir f:0 r:0.1
        \lvec(6 .7)\fcir f:0 r:0.1
        \lvec(7 .7)\fcir f:0 r:0.1
        \lvec(8 .7)\fcir f:0 r:0.1
        \lvec(9 .7)\fcir f:0 r:0.1
        \lvec(10 .7)\fcir f:0 r:0.1
        \lvec(11 .7)\fcir f:0 r:0.1
        \move(3 .7)\fcir f:0 r:0.1
        \lvec (2 0)
        \fcir f:0 r:0.1
         \htext(.3 .2){\footnotesize $a$}
         \htext(1.4 .2){\footnotesize $c$}
         \htext(1.2 1.1){\footnotesize $b$}
         \htext(1.1 -1.2){\footnotesize $d$}

     \end{texdraw}     
\end{center}
In fact, it is straightforward to embed this graph in a plane $H$ so that it is balanced and the quadrivalent vertex is indecomposable.  One may add pairs of unbounded edges to the boundary of this graph in $\R^3$ to ensure that it is balanced.  This example does not satisfy Nishinou's one-bouquet condition which is required for his higher genus necessary and sufficient condition to apply.

We claim that such a graph cannot  be the tropicalization of a map of a smooth curve of genus $3$.  Since each edge is given multiplicity $1$ and each vertex is indecomposable, the only parameterization of this graph is the identity.  Moreover, any genus $3$ lift of this curve must be maximally degenerate.
 Let us suppose that we have a function $\varphi_m$ meeting the conditions of Theorem \ref{maintheorem}.  We now apply the inequality,
 \[\sum_e s_{\varphi_m}(v,e)\leq\deg(v)-2\]
 of Lemma \ref{slopesums}.
We orient the edges not part of any cycle towards the nearest cycle. 
Because the slope of $\varphi_m$ on the two edges pointing towards edge $a$ is at most $1$,
slope of $\varphi_m$ on $a$ is at most $3$.  The slopes of $\varphi_m$ on $b$,$c$, and $d$ must sum to at most $5$.  Consequently, we may suppose that $\varphi_m$ has slope at most $1$ on edge $b$. Let $\Gamma$ be the cycle intersecting $b$.  The other paths from $\Gamma$  to $\partial\Gamma'$ are too long for $\varphi_m$ to have positive slope on them and for $\varphi_m|_\Gamma$ to be minimized along their intersection with $\Gamma$.  It follows that they cannot contribute to $D_{\varphi_m}$, and $D_{\varphi_m}$ must have degree at most $1$ on $\Gamma$.  
Consequently, $D_{\varphi_m}$ cannot have a non-constant section on $(C_\Gamma)_0$ for $\Gamma$, and $\varphi_m$ cannot be $\cc_0$-ample on $\Gamma$ in $h^{-1}(H)$.

\section{Toric Schemes, Log Structures, and Tropicalization} \label{s:tslst}

In this section, we review  the construction of the toric scheme $\cp$ over $\O$ from a rational polyhedral subdivision $\Xi$ of $\R^n$ \cite{KKMS,NS,Spe05}, background about log structures \cite{FKDef,KKLog}, and a suitable notion of parameterized tropicalization \cite{NS}.

Recall that $\O=\C[[t]]$.  
 Let $(\K^*)^n$ be an algebraic torus and $M=\Hom((\K^*)^n,\K^*)$ be its character lattice and $N=\Hom(\K^*,(\K^*)^n)$ be its one-parameter subgroup lattice.  Let $M_\R$ and $N_\R$ be $M\otimes\R$ and $N\otimes\R$, respectively.

\begin{definition} A complete rational polyhedral complex in $\R^n$ is a collection 
$\Xi$ of finitely many convex rational polyhedra
$P \subset \R^n$ whose minimal faces are vertices in $\Q^n$ such that 
\begin{enumerate}
\item If $P \in \Xi$ and $P^{\prime}$ is a face of $P$, then $P^{\prime}$
is in $\Xi$,
\item If $P,P^{\prime} \in \Xi$ then $P \cap P^{\prime}$ is a face of both
$P$ and $P^{\prime}$, and
\item The union $\bigcup_{P\in\Xi} P$  is equal to $\R^n$.
\end{enumerate}
\end{definition}

Given a $\Xi$ as above, we can construct a fan $\tXi$ in 
$\R^n \times \R_{\geq 0}$ as follows: for each $P \in \Xi$ let
$\tilde{P}$ be the closure in $\R^n \times \R_{\geq 0}$ of the set
$$\{(x,a) \subset \R^n \times \R_{>0} : \frac{x}{a} \in P\}.$$
Then $\tilde{P}$ is a rational polyhedral cone in $\R^n \times \R_{\geq 0}$.  Its 
facets come in two types:
\begin{enumerate}
\item cones of the form $\tilde{P}^{\prime}$, where $P^{\prime}$ is a facet of $P$, and
\item the cone $P_0 = \tilde{P} \cap (\R^n \times \{0\})$.
\end{enumerate}
We let $\tXi$ be the collection of cones of the form $\tilde{P}$ and $P_0$ for $P$ in $\Xi$.
It is a rational polyhedral fan in $\R^n \times \R_{\geq 0}$ by Corollary 3.12 of \cite{BGS}.  Note that
$\Xi = \tXi \cap (\R^n \times \{1\})$.  Let the fan $\Xi_0$ be given by $\Xi_0=\tXi \cap (\R^n \times \{0\})$.

Let $X(\tXi)$ be the toric variety associated to the fan $\tXi$.  
Projection from $\R^n \times \R_{\geq 0}$ to $\R_{\geq 0}$ induces a map of 
fans from $\tXi$ to the fan
$\{0,\R_{\geq 0}\}$ associated to $\A^1$.  This gives rise to a flat morphism of
toric varieties $X(\tXi) \rightarrow \A^1$.   Let $\iota: \Spec \O \rightarrow \A^1$
be the inclusion induced by $\Z[t]\rightarrow \O$.
We will use $\cp=X(\Xi)$ denote the scheme over $\O$ given by $X(\tXi)\times_{\A^1}\O$.

We summarize results of~\cite{NS} concerning this construction:
\begin{enumerate}
\item The general fiber $X(\tXi) \times_{\Spec \O} \Spec \K$ is isomorphic
to the toric variety over $\K$ associated to $\Xi_0$.
\item If $\Xi$ is {\em integral}, i.e. the vertices of every polyhedron in $\Xi$
lie in $\Z^n$, then the central fiber $X(\tXi)_0 = X(\tXi) \times_{\Spec \O} \Spec \k$
is reduced.   
\item There is an inclusion-reversing bijection
between closed torus orbits in $X(\tXi)_\k$ and polyhedra $P$ in $\Xi$; the irreducible 
components of $X(\tXi)_\k$ correspond to vertices in $\Xi$; the intersection of a collection
of irreducible components corresponds to the smallest polyhedron in $\Xi$ containing all
of their vertices.
\end{enumerate}

If $\cp$ is reduced, $\cp_0$ is a union of toric strata corresponding to the subdivision $\Xi$.  For a cell $P\in\Xi$, $\cu_P=\Spec \Z[\tilde{P}^\vee]\times_{\A^1}\O$ is a toric open set of $\cp$.   Let $\partial\cp_\K$ be the union of proper torus orbit closures of $\cp_\K$.

We now introduce log structures closely following the exposition of \cite{FKDef}.

\begin{definition} A {\em pre-log structure} on a scheme $X$ is a pair $(\cm,\alpha)$ where $\cm$ is a sheaf of monoids $\cm$ and $\alpha$ is a homomorphism $\alpha:\cm\rightarrow\O_X$.  If $\alpha$ induces an isomorphism $\alpha^{-1}(\O_X^*)\cong \O_X^*$, then we say $(\cm,\alpha)$ is a {\em log structure}.   A {\em log scheme} $(X,\cm,\alpha)$ is a scheme $X$ with a log structure $(\cm,\alpha)$.  We may denote such a log scheme by $X^\dagger$.
\end{definition}

A pre-log structure $(\cm,\alpha)$ on $X$ canonically induces a log structure $(\cm^a,\alpha^a)$ on $X$ by adjoining units to $\cm$.    Let $\cm^{\gp}$ be the sheaf of groups formed by groupifying the sheaf of monoids $\cm$.  We call sections of $\cm^{\gp}$ {\em generalized units}.

\begin{definition} A {\em morphism of log schemes} $f:(X,\cm,\alpha)\rightarrow(Y,\cn,\beta)$ is a pair $(f,\phi)$ where $f$ is a morphism of schemes $f:X\rightarrow Y$ and $\phi$ is a homomorphism of sheaves of monoids on $X$, $\phi:f^{-1}\cn\rightarrow\cm$ such that the following diagram commutes:
\[\xymatrix{
f^{-1}\cn\ar[d]\ar[r]^\phi&\cm\ar[d]\\
f^{-1}\O_Y\ar[r]&\ \O_X.
}\]
\end{definition}
Log structures glue in a way similar to schemes.

\begin{example} The trivial log structure on $X$ is $\cm=\O^*_X$ with $\alpha=1_{\O_X}$.
\end{example}

\begin{example} For a commutative ring $A$ and a monoid $P$, there is a natural pre-log structure on $\Spec A[P]$ given by $\alpha:P\rightarrow A[P]$. The induced log scheme $(\Spec A[P],P^a,\alpha^a)$ is called the {\em canonical monoid log structure}.  The construction is functorial: if $\phi:P\rightarrow Q$ is a homomorphism of monoids, there is an induced morphism of log schemes $\phi^*:(\Spec A[Q],Q^a)\rightarrow (\Spec A[P],P^a)$.  
\end{example}

\begin{example} For $\Delta$, a rational fan in $\R^n$, there is a natural log structure on the toric variety $X(\Delta)$.  For each cone $\sigma\in\Delta$, there is a toric affine open $U_\sigma=\Spec \k[M\cap\sigma^\vee]$ with the monoid log structure.  These log structures glue to give a log structure on $X(\Delta)$.
\end{example}

\begin{example} Since $\A^1$ is the toric variety associated to $\R_{\geq 0}$, it naturally has a log structure induced by $\alpha:\N\rightarrow \k[t]$ where $\alpha(n)=t^n$.  By considering $\Spec \O$ as the formal local neighborhood of the origin of $\A^1$, we obtain a log structure on $\Spec \O$  given by
$\alpha:\O^*\times\N \rightarrow \O$ where $\alpha(c,n)=ct^n.$
\end{example}

\begin{example} \label{localmodelformarkedpoints}
Given a smooth point $p$ on a curve $C$, there is a log structure on $C$ that is trivial away from $p$ and near $p$ is induced by $\N\rightarrow\O_C$ taking $n\mapsto u^n$ where $u$ is a uniformizer for $p$.  This is called the {\em model structure for marked points.}
\end{example}

\begin{example} If $\Xi$ is a rational polyhedral complex in $\R^n$ then $X(\tXi)$ has a log structure.  Moreover, the natural map $X(\tXi)\rightarrow \Spec \O$ is a morphism of log schemes.
\end{example}

\begin{example} \label{localmodelfornodes} Pick $l\in\N$ and let $P=(0,l)\subset\R$.  Set $\sigma=\tilde{P}\subset\R\times\R_{\geq 0}$.  Let $\pi:\R\times\R_{\geq 0}\rightarrow \R_{\geq 0}$ be projection on the second factor. The cone $\sigma^\vee$ is spanned by the vectors $(-1,l),(1,0)$.  The monoid $\sigma^\vee\cap M$ is generated by the elements $f_1=(-1,l),f_2=(1,0),e=(0,1)$ under the relation $f_1+f_2=l\cdot e$.  If we write these generators $x_1,x_2,t$ then we have an explicit description of the morphism $\pi^\vee:\k[\R_{\geq 0}^\vee]\rightarrow \k[\sigma^\vee\cap M]$ as the inclusion $\k[t]\hookrightarrow\k[x_1,x_2,t]/(x_1x_2-t^l)$.  The log morphism is induced by 
\[\xymatrix{
\N e\ar[r]\ar[d]&\N f_1\oplus\N f_2\oplus\N e/(f_1+f_2=l\cdot e)\ar[d]\\
\k[t]\ar[r]&\k[x_1,x_2,t]/(x_1x_2-t^l).
}\]
By base-changing $\k[t]$ to $\O$, we get a formal local model for nodes over $\O$.  This log structure is the {\em model structure for nodes}.
\end{example}

A marked semistable family of curves has a canonical log structure that is trivial away from nodes and marked points and has the respective model structures near nodes and marked points.

\begin{definition} Let $f:X^\dagger=(X,\cm)\rightarrow Y^\dagger=(Y,\cn)$ be a morphism of log schemes.  The sheaf of log differentials of $X^\dagger$ over $Y^\dagger$ is
\[\Omega^1_{X^\dagger/Y^\dagger}=[\Omega^1_{X/Y}\oplus (\O_X\otimes_\Z \cm^{\gp})]/\ck\]
where $\ck$ is the $\O_X$-submodule generated by
\[(d\alpha(a),0)-(0,\alpha(a)\otimes a)\ \text{and}\ (0,1\otimes\phi(b))\]
for all $a\in\cm$ and $b\in f^{-1}\cn$.
\end{definition}

One should view log differentials as adjoining to the ordinary differentials elements of the form $\DLog(a)=\frac{\text{d}(\alpha(a))}{\alpha(a)}$ for $a\in\cm$.  Both $\cp^\dagger\rightarrow\O^\dagger$ and  $\cc^\dagger\rightarrow \O^\dagger$ are log smooth morphisms \cite[Ex 4.6,4.7]{FKDef}.  Consequently, the log differentials $\Omega^1_{\cp^\dagger/\O^\dagger}$ and $\Omega^1_{\cc^\dagger/\O^\dagger}$ are locally free sheaves.

\begin{example}
Let $\cc$ be a regular semistable family over $\O$.  Give $\cc$ the model log structure near nodes and marked points.  We consider the log differentials in $\Omega^1_{\cc^\dagger/\O^\dagger}$.
If $u$ is a uniformizer of a marked point, the log differentials include the $1$-form $\frac{\text{d}u}{u}$.  If $\O[x_1,x_2]/(x_1x_2-t)$ is a local model for a node in $\cc$, we have the log differentials in $\Omega^1_{\cc^\dagger/\O^\dagger}$ given by $\frac{\text{d}x_1}{x_1},\frac{\text{d}x_2}{x_2}$ subject to  
$\frac{\text{d}x_1}{x_1}+\frac{\text{d}x_2}{x_2}=0$.  In fact, $\Omega^1_{\cc^\dagger/\O^\dagger}$ is an invertible sheaf.  Near marked points and nodes, it is generated by $\frac{\text{d}u}{u}$ and $\frac{\text{d}x_1}{x_1}$, respectively.  
One can view log differentials on $\cc_0$ as $1$-forms that are allowed simple poles at marked points and simple poles at nodes such that the residues on the branches sum to $0$.  In fact, if $\cc$ is a marked semistable curve, $\Omega^1_{\cc^\dagger/\O^\dagger}$ is the relative dualizing sheaf twisted by the divisor of marked points.  Nodes and marked points of the special fiber $\cc_0$ are called {\em special points}.  

The specialization of the invertible sheaf $\Omega^1_{\cc^\dagger/\O^\dagger}$ to the dual graph is particularly important.   Let $\Sigma$ be the dual graph of $\cc$.  Now, $\Omega^1_{\cc^\dagger/\O^\dagger}$ pulls back to a component $C_v$ of the normalization $\widetilde{\cc}_0$  as an invertible sheaf of degree $\deg(v)+2g(v)-2$.  Consequently, $\Omega^1_{\cc^\dagger/\O^\dagger}$ specializes to 
\[K_\Sigma=\sum_v (\deg(v)+2g(C_v)-2)(v)\]
which is the canonical divisor.\end{example}

\begin{example}
Let $\cp=X(\Xi)$ be a toric scheme over $\O$.  For $m$ a character of $(\K^*)^n$, $\omega_m=\DLog(z^m)$ is a regular log differential. 
\end{example}

We need to make use of the following theorem due to Nishinou-Siebert \cite{NS} about completing families of maps of curves.

\begin{definition} Let $X(\Delta)$ be a toric variety.  A stable map $f:C\rightarrow X(\Delta)$ is {\em torically transverse} if $f:C\rightarrow X(\Delta)$ satisfies
\begin{enumerate}
\item $f^{-1}((\G_m)^n)\subset C$ is dense, and
\item $f(C)\subset X(\Delta)$ is disjoint from strata of codimension greater than $1$.
\end{enumerate}
\end{definition}

\begin{theorem}\cite{NS} \label{goodmodel} Let $f:C^*\rightarrow (\K^*)^n$ be a map of a smooth curve to an algebraic torus.   Then after a possible base-change $\O[t^{\frac{1}{N}}]\rightarrow \O$, there is a completion of $(\K^*)^n$ to a toric scheme $\cp=X(\Xi)$, a completion of $C^*$ to a proper stable family $\cc$ over $\O$, and an extension $f:\cc\rightarrow\cp$ such that
\begin{enumerate}
\item $f_0:\cc_0\rightarrow\cp_0$ has the property that for every irreducible component $\cp'_0\subset\cp_0$, $f_0:f_0^{-1}(\cp'_0)\rightarrow\cp_0$ is a torically transverse stable map.
 
\item There exists disjoint sections $\sigma_1,\dots,\sigma_k:\Spec \O\rightarrow\cc^{\sm}$ such that $f^{-1}(\partial\cp_{\K})=\bigsqcup \sigma_i(\Spec \K)$.

\item Near $\sigma_{e'}(\Spec \k)$, $\cp$ is formal locally modeled on $(\G_m)^{n-1}\times\A^1_\O$ and $f$ is modeled on $z\mapsto cu^{\mu(e')}$ where $z$ is a uniformizer for $\A^1$ at $0$, $u$ is a uniformizer for $\sigma_{e'}(\Spec \O)$ in $\cc$, $\mu(e')\in\N$ and $c$ is a unit.

\item \label{node} Near the intersection of two irreducible components of $\cp_0$ that is equal to $f(p_e)$
 for a node $p_e\in\cc_0$, $\cp$ is formal locally modeled on 
 $(\G_m)^{n-1}\times\left(\Spec \O[w_1,w_2]/(w_1w_2-t^{s(e)})\right)$ for $s(e)\in\N$, $\cc$ is modeled near $p_e$ on 
 $\Spec \O[x_1,x_2]/(x_1x_2-ct^{s(e)/\mu(e)})$, 
 $\mu(e)\in \N$ and $f$ is modeled on $f^*w_i=c_ix_i^{\mu(e)}$ where $\mu(e)\in\N$ and $c_i$ is a 
 unit.
\end{enumerate} 
Moreover, if $\cc$ is given the canonical log structure for marked semistable families, $f:\cc\rightarrow\cp$ is a log morphism where $\cp$ is given the log structure of a toric scheme.
\end{theorem}

\begin{proof}
One uses Proposition 6.3 of \cite{NS} to extend the map $C^*$ to $\cc$.  The local models are produced in the proof of Theorem 8.3 of \cite{NS}.
\end{proof}

The base-change corresponds to rescaling $\R^n$ such that the vertices of $\Trop(f(C^*))$ have integral coordinates.
Here, $\Xi$ will be chosen to be a complete polyhedral subdivision of $\R^n$ such that $\Trop(f(C^*))$ is a union of polyhedra in $\Xi$.  Such a $\Xi$ exists by Theorem 2.2.1 of \cite{Spe05}.  The intersection of two irreducible components of $\cp_0$ in (\ref{node}) above corresponds to an edge of lattice length $s(e)$ in $\Xi$.  

This allows us to construct a map of the dual graph $\Trop(f):\Sigma\rightarrow N_\R=\R^n$ that we call the {\em parameterized tropicalization}.  This is essentially a rephrasing of Construction 4.4 of Nishinou-Siebert \cite{NS} and is also developed by Tyomkin \cite{Tyomkin}.  Pick a vertex $v_0\in V(\Sigma)$.  Let $x$ be a $\K$-point of $C^*$ specializing to a smooth point on the component $C_{v_0}$ in the central fiber.  Set $\Trop(f)(v_0)=\val(f(x))\in N_\R$ where $\val:(\K^*)^n\rightarrow N_\R$ is the valuation.
For $e\in E(\Gamma)^\bullet$ such that $f(p_e)$ is mapped to a smooth point of $\cp_0$, set $\Trop(f)$ to be constant on $e$.  For a bounded edge $e\cong[0,1]$ from $v_1$ to $v_2$ that is mapped to a singular point of $\cp_0$, we have a map $M\rightarrow\Z$
given by $m\mapsto\res_{p_e}(f^*{\frac{dz^m}{z^m}}\big|_{C_{v_1}})$.
This gives an element $n(e)\in N_\R$.  Define $\Trop(f)$ on $e$ by 
\[\Trop(f)(t)=\Trop(f)(v_1)+\left(\frac{s(e)}{\mu(e)}n(e)\right)t.\]
In the case where $\Trop(f)$ is constant on an edge $e$, we say the edge is {\em contracted}.  Note that otherwise, the vector along the edge in the parameterized tropicalization is the scalar multiple of the vector $n(e)$ by $\frac{s(e)}{\mu(e)}$ which is the thickness of the node.
Let $\sigma_{e'}$ be a marked point corresponding to $e'\cong [0,\infty)$ such that $\sigma_{e'}(\k)\in C_v$.  We have a similar map $M\rightarrow \R$ taking the residue of $f^*{\frac{dz^m}{z^m}}$ at $\sigma_{e'}$.  This gives $n(e')\in N$.  Define $\Trop(f)$ on $e'$ by
\[\Trop(f)(t)=\Trop(f)(v)+\frac{1}{\mu(e)}n(e')t.\]
This map is well-defined by as it constructs $\Trop(f)(\Sigma)$ as supported on $\Xi$.  In fact, the image of the parameterized tropicalization is the tropicalization of the curve $f(C^*)$, $\Trop(f)(\Sigma)=\Trop(f(C^*))$.

\begin{lemma} \label{l:parbal} $\Trop(f)$ satisfies the following balancing condition: if $v$ is a vertex of $\Sigma$ with bounded edges $e_1,\dots,e_k$ and unbounded edges $e'_1,\dots,e'_l$ then
\[\sum_{j=1}^k \frac{\mu(e_j)}{s(e_j)}(\Trop(f)|_{e_j}(1)-\Trop(f)|_{e_j}(0))+\sum_{j=1}^l \mu(e'_j)(\Trop(f)|_{e'_j}(1)-\Trop(f)|_{e'_j}(0))=0.\]
\end{lemma}

\begin{proof}
The quantity on the left is an element of $N_\R$.  Its evaluation on $m\in M$, is the sum of residues of $f^*{\frac{dz^m}{z^m}}$ on $C_v$.  This vanishes by the residue theorem.
\end{proof}

In the case where $f:C^*\rightarrow (\K^*)^n$ is a closed immersion and all the initial degenerations are reduced, this reproduces Speyer's notion of parameterized tropical curves \cite{Spe07}.

We perform $\left(\frac{s(e)}{\mu(e)}-1\right)$ blow-ups at $p_e$ to ensure that $\cc$ is a regular semistable model. This has the effect of subdividing $e$ into $s(e)/\mu(e)$ edges.  Likewise, we may also have to blow up nodes that are mapped to smooth points of $\cp$ and then subdivide to ensure a regular model.   In any case, we can ensure that any component of $\cc_0$ that is contracted by $f$ has at least two special points.

We give $\Sigma$ the structure of an abstract genus-marked weighted graph.  For  a vertex $v\in\Sigma$, set $g(v)=g(C_v)$.  Give an edge $e$ in $\Sigma$ multiplicity $\mu(e)$ if it is not contracted.  Give the edge multiplicity $0$ if it is contracted. Then, $\Trop(f)(e)$ points in the primitive integer direction $n(e)/\mu(e)$.  The weighted structure on $\Sigma$ is compatible with that of $\Trop(f(C^*))$ in the following sense:

\begin{lemma} The map $\Trop(f):\Sigma\rightarrow\Trop(f(C^*))$ is a tropical parameterization.
\end{lemma}

\begin{proof}
Condition (\ref{par:length}) is by construction.
Condition (\ref{par:bal}) is Lemma \ref{l:parbal}.  Condition (\ref{par:mult}) follows from the definition of the multiplicity of an edge of a tropical curve as the length of the associated initial degeneration \cite{Spe05}.  To prove condition (\ref{par:semistable}), note that any vertex $v\in V(\Sigma)$ with only contracted edges corresponds to a component $C_v$ of $\cc_0$ on which $f$ is constant.  By construction, we ensured that $C_v$  contains at least two special points.  
\end{proof}

\section{Specialization of Sections of Line Bundles}

In this section, we will let $\cul$ be a line bundle over a regular semistable family $\cc$ with dual graph $\Sigma$.  Let $s$ a rational section of $\cul$.  To $s$, we will associate a piecewise linear function $\varpi_s$ on $\Sigma$, called a vanishing function.  

\begin{definition} For $s$ a rational section of $\cul$, we say that $s$ has {\em $\K$-rational zeroes and poles} if the divisor $(s)$ on $C_{\overline{\K}}$ is supported on $C(\K)$.
\end{definition}

Note that $\K$-rational points of $\cc$ specialize to smooth points of the central fiber $\cc_0$.  Let $\pi:\widetilde{\cc}_0\rightarrow \cc_0$ be the normalization morphism.  We use $C_v$ to refer to a  component of $\tilde{C}_0$.  

\begin{definition}    Let $s$ be a rational section of $\cul$ with $\K$-rational zeroes and poles.  Define the {\em vanishing function} of $s$ to be $\varpi_s:\Sigma\rightarrow\R\cup\{\infty\}$ by setting $\varpi_s(v)$ to be the multiplicity of $C_v$ in the divisor $(s)$.
For $v\in V(\Sigma)$, let $s_v=\pi^*\left(\frac{s}{t^{\varpi(v)}}\right)|_{C_v}$.  For $e\in E(\Sigma)$ adjacent to a vertex $v$, let $\ord_{p_e}(s_v)$ be the order of vanishing of $s_v$ at $p_e\in C_v$. 
  We extend $\varpi_s$ linearly on bounded edges.  If $e$ is an unbounded edge adjacent to a vertex $v$, set the slope of $\varphi_s$ on $e$ away from $v$ to be equal to $\ord_{p_e}(s_v)$.  If $s=0$, set $\varpi_s=\infty$.
\end{definition}

Note that if we apply a base-change $\O\rightarrow \O[t^{\frac{1}{l}}]$ and blow-up the nodes to produce the model $\cc^l$, the vanishing function changes in a predictable way.  If $\tau:\cc^l\rightarrow\cc$, and we fix a homeomorphism $\Sigma\cong\Sigma_l$ that rescales the edges by a factor of $l$, then $\varpi_{\tau^*s}=l\varpi_{s}$

The divisor of $s$ as a section of $\cul$ satisfies
\[(s)=\overline{D}+\sum \varpi_s(v)C_v\]
where $\overline{D}$ is a horizontal divisor.
The vanishing function can therefore be seen as encoding the vanishing of $(s)$ on components of the central fiber together with the components of $\overline{D}$ that specialize to marked points $p_{e'}$ in the central fiber.
Note that $\varpi_s(v)$ is the order of vanishing of $s$ on the generic point of $C_v$ in the central fiber.  If $s$ is a regular section of $\cul$, then $\varpi_s\geq 0$.  In the case where $s$ is a section of the restriction of $\cul$ to the generic fiber, we define $\varpi_s$ by extending $s$ as a rational section on $\cc$.  

We now relate the poles of $s_v$ to $\varpi_s$.  

\begin{lemma} \label{slopes} Suppose $s$ has $\K$-rational zeroes and poles.  Let $e\in E(\Sigma)$ be an edge adjacent to $v\in V(\Sigma)$.  Then $\ord_{p_e}(s_v)$ is equal to the slope of $\varpi_s$ along $e$ away from $v$.  Consequently, $\varpi_s$ has integer slopes.
\end{lemma}

\begin{proof}
For $e$, an unbounded edge, the statement follows by definition.  For a bounded edge
$e=v_1v_2$, we must show
\[\ord_{p_e}(s_{v_1})=\varpi_s(v_2)-\varpi_s(v_1).\]
Write $(s)=\overline{D}+\sum_v \varpi(v)C_v$ on $\cc$ where $\overline{D}$ is a horizontal divisor.
Then $\frac{s}{t^{\varpi_s(v_1)}}$ has divisor $\overline{D}+\sum_v (\varpi_s(v)-\varpi_s(v_1))C_v$. 
 Since $p_e$ is a node, no component of $\overline{D}$ intersects the central fiber in $p_e$. Consequently, near $p_e$, $\frac{s}{t^{\varpi_s(v_1)}}$ has the divisor 
 $(\varpi_s(v_2)-\varpi_s(v_1))C_{v_2}$.  By restricting to $C_{v_1}$, we see $\ord_{p_e}(s_{v_1})=\varpi_s(v_2)-\varpi_s(v_1)$.
\end{proof}

Let $\Lambda$ be the divisor on $\Sigma$ given by 
\[\Lambda=\sum_{v\in V(\Sigma)} \deg(\pi^*\cul|_{C_v})(v).\]

\begin{lemma} \label{positivity} If $s$ is a section of $\cul$ that is regular on the generic fiber $C$ and has $\K$-rational zeroes, then
$\Delta(\varpi_s)+\Lambda\geq 0$.  That is, $\varpi_s\in L(\Lambda)$.
\end{lemma}

\begin{proof}
Write $(s)=D$ for an effective divisor $D$ on $C$.  On $\cc$ we have
\[(s)=\overline{D}+\sum_v \varpi_s(v)C_v.\]
We check the inequality on each vertex $w$ of $\Sigma$.  The rational section $\frac{s}{t^{\varpi_s(w)}}$ does not vanish identically on $C_v$.  It has divisor 
\[\left(\frac{s}{t^{\varpi_s(w)}}\right)=\overline{D}+\sum_v (\varpi_s(v)-\varpi_s(w))C_v\]
which is linearly equivalent to $(s)$.  Therefore, $\Lambda(s)(w)=\rho((s))(w)=\rho\left((\frac{s}{t^{\varpi_s(w)}})\right)(w)$.
If $e'_1,\dots,e'_k$ are the unbounded edges adjacent to $v_1$, we may decompose the intersection product
\[\overline{D}\cdot C_{w}=D_{w,\sm}+\sum_i m_i p_{e'_i}\]
where $D_{w,\sm}$ is an effective divisor supported on unmarked smooth points of $C_w$.  Consequently, since $m_i=\ord_{p_{e'_i}}(s_w)$ and $\varpi_s(v)-\varpi_s(w)=\ord_{p_e}(s_w)$ for $e=vw$, we have
\begin{eqnarray*}
\rho\left((\frac{s}{t^{\varpi_s(w)}})\right)(w)&=&\deg((\overline{D}+\sum_v ((\varpi_s(v)-\varpi_s(w))C_v)\cdot C_w)\\
&=&\deg(D_{w,\sm})+\sum_i m_i+\sum_{e=vw}(\varpi_s(v)-\varpi_s(w))\\
&=&\deg(D_{w,\sm})-\Delta(\varpi_s)(w)\geq -\Delta(\varpi_s)(w).
\end{eqnarray*}
\end{proof}

This above is lemma is analogous to a step in the proof of the specialization lemma in \cite{B}.  It is also closely related to the Poincar\'{e}-Lelong formula on Berkovich curves \cite{BPR}.

We will need the following lemma to find an algebraic extension $\K'$ of $\K$ to ensure that the zeroes of all elements of a linear system are $\K'$-rational.

\begin{lemma} \label{KRationalZeroes} Let $L$ be a line bundle on $C$ defined over $\K$.  Let $V=\Gamma(C,L)$ be the sections of $L$.  Then there exists a finite field extension $\K'/\K$ such that all zeroes of any non-zero  $s\in V_\K$ are $\K'$-rational.
\end{lemma}

\begin{proof}
$L$ has finite degree on $C$, say $d$.  If $s\neq 0$, $(s)_{\overline{\K}}\subset C(\overline{\K})$ consists of at most $d$ points.  This gives a homomorphism 
\[\hat{\Z}=\Gal(\overline{\K}/\K)\rightarrow \Aut((s)_{\overline{\K}}).\]
Since each Galois-orbit of a zero of $s$ has at most $d$ elements, 
every element of $\Aut((s)_{\overline{\K}})$ has order dividing $d!$.  Consequently, the 
subgroup $d!\hat{\Z}\subset\hat{\Z}$ acts trivially on $(s)_{\overline{\K}}$.
 But $d!\hat{\Z}$ has fixed field $\K'=\K[t^\frac{1}{d!}]$.  Consequently, the Galois group $\Gal(\overline{\K}/\K')$ acts trivially on $(s)_{\overline{\K}}$.  It follows that the zeroes of $s$ are $\K'$-rational.  This choice of $\K'$ was independent of $s$.  
\end{proof}

In general, given any field $\K$ and a finite number of sections, we can set $\K'$ to be a field containing the fields of definition of the zeroes of the sections.  In the above lemma, we have infinitely many sections and had to make use of the fact that $\K=\C((t))$.

\begin{example} Consider a map $f:C^*\rightarrow (\K^*)^n$ to an algebraic torus.  By applying Theorem \ref{goodmodel}, we may produce a model $\cc$ completing $C^*$ and an extension $f:\cc\rightarrow\cp$ mapping to a toric scheme.  This induces a parameterized tropicalization $\Trop(f):\Sigma\rightarrow N_\R$.
 For $m\in M$, the pullback of the character $f^*z^m$ is a rational function on $\cc$.  Let $\psi$ be the vanishing function of $f^*z^m$ considered as a section of the trivial bundle.  
By construction, the only zeroes and poles of $z^m$ are at the marked points $\sigma_i$ which are $\K$-rational.   We claim that 
 \[\psi(t)=\<m,\Trop(f)(t)\>\]
 for all $t\in\Sigma$.   
$\Trop(f)$ is first defined with reference to an index vertex $v_0$.  The vanishing of $z^m$ on $C_{v_0}$ is exactly $\<m,\val(f(x))\>$ for a $\K$-point $x$ specializing to a smooth unmarked point of $C_v$.  Now, we check that $\psi(t)$ and $\<m,\Trop(f)(t)\>$ agree on each edge.  The slope of $\<m,\Trop(f)(t)\>$ on an edge $e$ at $v$ is the residue of $f^*\frac{dz^m}{z^m}|_{C_v}$ at $p_e$.  This is the order of vanishing of $(f^*z^m)_v=\frac{z^m}{t^{\psi(v)}}|_{C_v}$ at $p_e$.  This, in turn, is equal to the slope of $\psi(t)$ by Lemma \ref{slopes} for bounded edges and by definition for unbounded edges.
\end{example}

We now relate the vanishing functions of two sections and their sum.  
Recall that $\Sigma^\bullet$ is the subgraph of $\Sigma$ consisting of all vertices and bounded edges.

\begin{lemma} \label{addingsections} Given sections $s_1,s_2$ such that $s_1,s_2,s_1+s_2$ have $\K$-rational zeroes  then the restrictions of the vanishing functions to $\Sigma^\bullet$ satisfy
$\varpi_{s_1}\oplus\varpi_{s_2}\oplus\varpi_{s_1+s_2}=0.$
\end{lemma}

\begin{proof}
We first check on vertices.  If $\varpi_{s_1}(v)\neq\varpi_{s_2}(v)$ then $\varpi_{s_1+s_2}(v)=\min(\varpi_{s_1}(v),\varpi_{s_2}(v))$.  If $\varpi_{s_1}(v)=\varpi_{s_2}(v)$ then $\varpi_{s_1+s_2}(v)\geq\min(\varpi_{s_1}(v),\varpi_{s_2}(v))$.  

Now, we check bounded edges.  Let $e=vw$ be a bounded edge.  Since $\varpi_s=\varpi_{-s}$, we may treat $\{s_1,s_2,s_1+s_2\}$ symmetrically.  First consider the case that $\varpi_{s_1}(v)=\varpi_{s_2}(v)=\varpi_{s_1+s_2}(v)$.  We may also suppose $\varpi_{s_1}(w)=\varpi_{s_2}(w)\leq\varpi_{s_1+s_2}(w)$.  Then the conclusion immediately follows.  Now suppose that $\varpi_{s_1}(v)=\varpi_{s_2}(v)<\varpi_{s_1+s_2}(v)$.  Then 
\[(s_1)_v+(s_2)_{v}=\left(\frac{s_1+s_2}{t^{\varpi_{s_1}(v)}}\right)\Big|_{C_v}=0.\]
Therefore, $\ord_{p_e}((s_1)_v)=\ord_{p_e}((s_2)_v)$.  By Lemma \ref{slopes}, $\varpi_{s_1}(w)=\varpi_{s_2}(w)$.  Since the result is true on vertices, $\varpi_{s_1}(w)=\varpi_{s_2}(w)\leq \varpi_{s_1+s_2}(w)$.
\end{proof}

An analogous statement holds for the vanishing functions on $\Sigma$ not just their restrictions to $\Sigma^\bullet$ but this requires choosing a model $\cc$ such that the only zeroes of $s$ that specialize to the points $\sigma_i(\Spec \k)$ of  $\cc_0$ are contained in the sections $\sigma_i(\Spec \O)$.  This model would depend on $s$.  For the sake of picking a single model, we have chosen just to consider the restriction to $\Sigma^\bullet$ . 

\section{Vanishing functions of log differentials}

Let $(\cc,\cm)$ be a regular semistable family of curves with the canonical log structure over $\O$.   Given a section $s\in\Gamma(\cc,\cm^{\gp})$ which we call a generalized unit, one can study the vanishing function of the log differential $\DLog(s)$ considered as a section of the sheaf of log differentials.  This vanishing function is tightly constrained and, in turn, imposes conditions on the vanishing function of $s$.

We will work under the following assumptions:
\begin{enumerate}
\item $s$ is a section of $\cm^{\gp}$, and
\item $\DLog(s)$ has $\K$-rational zeroes.
\end{enumerate}

Let $\psi$ be the vanishing function of $s$ considered as a rational section of the trivial bundle on $\cc$, and let $\varphi$ be the vanishing function of $\omega=\DLog(s)$ considered as a section of $\Omega^1_{\cc^\dagger/\O^\dagger}$.  By Lemma \ref{slopes}, $\varphi$ is an element of $L(K_\Sigma)$.

\begin{lemma} \label{c0ample} Let $c\in\R$.  Let $\Gamma'=\psi^{-1}(c)$ and $\Gamma$ be a bounded, connected subgraph contained in the interior of $\Gamma'$ (considered as a subspace of $\Sigma$).  Then $\varphi$ is $\cc_0$-ample on $\Gamma$ in $\Gamma'$.
\end{lemma}

\begin{proof}
Let $(\mathring{\cc_{\Gamma'}})_0=\cc_0\setminus\cup_{v\notin\Gamma'} C_v$, the points that are contained only in curves $C_v$ for $v\in\Gamma'$. 
 Let $(\mathring{\cc}_{\Gamma'})_k$ be the $k$th order thickening of $(\mathring{\cc}_{\Gamma'})_0$ in $\cc_k=\cc\times_{\Spec \O} \Spec \O/(t^{k+1})$ given by pulling back the structure sheaf of $\cc_k$.
 On $(\mathring{\cc}_{\Gamma'})_0$, $t^{-c}s$ is a unit.  In fact, it is invertible near non-special points since there $\cm^{\gp}$ is the sheaf of units.  At special points of $C_v$ contained in $(\mathring{\cc_{\Gamma'}})_0$, $s_v=t^{-c}s|_{C_v}$ has neither poles nor zeroes by Lemma \ref{slopes}.
 Consequently, $t^{-c}s$ is a unit on $(\mathring{\cc}_{\Gamma'})_k$ for all $k\geq 0$.
 
 Since $t^{-c}s$ is a unit on $(\mathring{\cc}_{\Gamma'})_0$, it is constant on complete components of $(\mathring{\cc}_{\Gamma'})_0$ .  These are of the form $C_v$ for $v\in\Gamma'\setminus\partial\Gamma'$ where $\partial \Gamma'$ is the set of vertices in $\Gamma'\cap \overline{\Sigma\setminus\Gamma'}$.
Since $\omega=\frac{\text{d}t^{-c}s}{t^{-c}s}$ is the log differential of a constant function, it vanishes on $C_v$ for $v\in\Gamma\subset\Gamma'\setminus\partial\Gamma'$.
 If we let  $h=\min_{v\in\Gamma}\varphi(v)$, we must have $h>0$.  
 
 Now, we will subtract an appropriate constant from $t^{-c}s$ and divide by $t^h$ to find a rational function on $(\cc_\Gamma)_0$ that has poles of the same order as $\frac{\omega}{t^h}$.
This will give the desired section of $\O_{(\cc_\Gamma)_0}(D_\varphi)$.
Note that $\omega$ vanishes on $(\mathring{\cc}_{\Gamma})_{h-1}$, that is, $\omega$ vanishes to the $(h-1)^{\text{st}}$ order on components $C_v$ with $v\in\Gamma$. Therefore, $t^{-c}s$  is equal to a constant $\overline{L}\in\O_{h-1}$ on the completion, $(\cc_\Gamma)_{h-1}$.  Lift $\overline{L}$ to some $L\in\O_h$.    From the exact sequence
\[\xymatrix{0\ar[r]&t^h\O_{(\mathring{\cc}_{\Gamma})_{h}}\ar[r]& \O_{(\mathring{\cc}_{\Gamma})_{h}}\ar[r]& \O_{(\mathring{\cc}_{\Gamma})_{h-1}}\ar[r]&0},\]
$t^{-c}s-L$ is a section of 
 $t^h\O_{(\mathring{\cc}_{\Gamma})_{h}}$ .

Let $\tilde{s}=\frac{t^{-c}s-L}{t^h}$.  Extend $\tilde{s}$ as a rational function to each component of $(\cc_\Gamma)_0$.   
If $\varphi(v)>h$, then $\frac{\text{d}\tilde{s}}{t^{-c}s}=\frac{\omega}{t^h}=0$ on $C_v$ so $\tilde{s}$ is constant on $C_v$.  On the other hand, if $\varphi(v)=h$, then $\frac{\text{d}\tilde{s}}{t^{-c}s}=\frac{\omega}{t^h}$ is non-vanishing on $C_v$, and so $\tilde{s}$ is non-constant on $C_v$.
For $v\in\Gamma$ with $\varphi(v)=h$ and $e$, an edge in $\Gamma'$ containing $v$, from Lemma \ref{slopes}, we have $\ord_{p_e}(\frac{\omega}{t^h}|_{C_v})=s_{\varphi_m}(v,e)$.  If this quantity is negative, $\frac{\omega}{t^h}|_{C_v}$ considered as a log $1$-form has a pole at $p_e$.  Now $t^{-c}s$ is a constant on $C_v$, so from $\frac{\text{d}\tilde{s}}{t^{-c}s}=\frac{\omega}{t^h}$, we have that $\text{d}\tilde{s}$ has a pole of the same order.  But then,
 $\ord_{p_e}(\text{d}\tilde{s}|_{C_v})=\ord_{p_e}(\tilde{s}|_{C_v})$ where $\tilde{s}$ is considered as a rational function.  Consequently, $\ord_{p_e}(\tilde{s}|_{C_v})=s_{\varphi_m}(v,e)$.
 
In any case, for any $v\in\Gamma$ with $\varphi(v)=h$ and edge $e\in E(\Gamma')\setminus E(\Gamma)$ containing $v$, we have
\[\ord_{p_e}(\tilde{s}|_{C_v})\geq \min(s_{\varphi}(v,e),0),\]
and so  $\tilde{s}$ is a section of $\O_{(\cc_\Gamma)_0}(D_{\varphi})$.
\end{proof}

\begin{lemma} \label{vanishingonedges}
If $\psi$ is non-constant on an edge $e\in E(\Sigma)$ then $\varphi=0$ on $e$.
\end{lemma}

\begin{proof}
Suppose that $e$ is adjacent to a vertex $v$.  Then by Lemma \ref{slopes}, $s_v$ has a pole or zero at $p_e$.  Consequently, $\omega=\frac{\text{d} s}{s}$ which is equal to $\frac{\text{d}t^{-\psi(v)}s}{t^{-\psi(v)}s}=\frac{\text{d}s_v}{s_v}$ on $C_v$ is non-vanishing.
\end{proof}

\begin{lemma} \label{nonvanishingonedges} If $\psi$ is constant on an edge $e\in E(\Sigma)$ then $\varphi$ is non-constant on $e$.
\end{lemma}

\begin{proof}
Suppose $\varphi$ is constant on $e$.   Let $e$ be adjacent to a vertex $v$.   $e$ corresponds to a special point $p_e$ on $C_v$.   Set $\tilde{s}=\frac{t^{-c}s-L}{t^{\varphi(v)}}$
and  $\omega_v=\frac{\omega}{t^{\varphi(v)}}|_{C_v}$
 as in the proof of the Lemma \ref{c0ample}.  Now, $\tilde{s}$ is regular on $(\cc_e)_0$ and non-constant on $C_v$.
Since $\omega_{v}=\frac{\text{d}\tilde{s}}{t^{-c}s}|_{C_{v}}$,
$\omega_{v}$ does not have a simple pole at $p_e$ and so vanishes as a log $1$-form at $p_e$.  Therefore, $\ord_{p_e}(\omega_{v})\neq 0$ and the contradiction follows from Lemma \ref{slopes}.
\end{proof}

This section can be seen in terms of Berkovich curves as in \cite{BPR}.  One has a rational function $s$ on a curve $C^*$.  The function $\psi$ is $\log|s|$ on $(C^*)^{\an}$.  The vanishing function $\varphi$ is $\log|\omega|$ where absolute value is taken with respect to a metric on the canonical bundle.  The above conditions constrain $\varphi$ in terms of $\psi$.

\section{Proof of Lifting Condition}
In this section, we prove Theorem \ref{maintheorem} and Corollary \ref{c:istrop}.  

We first give the proof of Theorem \ref{maintheorem}.  Let $f:C^*\rightarrow (\K^*)^n$ be the map of a smooth curve.   Pick a model $\cc$ according to Theorem \ref{goodmodel}.  Let $\Trop(f):\Sigma\rightarrow N_\R$ be the parameterized tropicalization.
For $m\in M$, define a log $1$-form $\omega_m=f^*\frac{dz^m}{z^m}$ on $\cc$.  By employing Lemma \ref{KRationalZeroes}, we may replace $\K$ by a finite extension to ensure that each $\omega_m$ has $\K$-rational zeroes.  
The finite extension has the effect of rescaling $\Trop(f)$ by a factor of $l$ for some $l\in\N$.

Let $\varphi_m$ be the vanishing function of $\omega_m$ considered as a section of $\Omega^1_{\cc^\dagger/\O^\dagger}$.  Note that $\omega_m$ is the log differential of $f^*z^m$ whose vanishing function is $\psi(x)=\<m,\Trop(f)(x)\>$.
Since $\omega_m$ is regular on $\cc$, $\varphi_m$ is non-negative.

The properties of $\varphi_m$ in Theorem \ref{maintheorem} follow from the results in the previous two sections.   
 Since 
\[f^*\left(\frac{dz^{m_1+m_2}}{z^{m_1+m_2}}\right)=f^*\left(\frac{dz^{m_1}}{z^{m_1}}\right)+f^*\left(\frac{dz^{m_2}}{z^{m_2}}\right),\] the function $m\mapsto \varphi_m|_{\Sigma^\bullet}$ is a tropical homomorphism by Lemma \ref{addingsections}.  The fact that $\varphi_m\in L(K_\Sigma)$ follows from Lemma \ref{positivity}.  By Lemma \ref{vanishingonedges}, $\varphi_m=0$ on any bounded edge $e$ with $m\cdot e\neq 0$.  By Lemma \ref{nonvanishingonedges}, $\varphi_m$ does not have slope zero on any bounded edge $e$ with $m\cdot e=0$.  The $\cc_0$-ampleness statement is Lemma \ref{c0ample}.

The proof of Corollary \ref{c:istrop} is almost immediate.  Given a curve $C^*\subset (\K^*)^n$, let $f:C^*\rightarrow (\K^*)^n$ be the closed embedding.  The arguments above produce a tropical parameterization $p:\Sigma\rightarrow\Sigma'=\Trop(C^*)$.  

\bibliographystyle{plain}

\end{document}